\documentclass[11pt]{amsart}
\usepackage{geometry}                % See geometry.pdf to learn the layout options. There are lots.
\usepackage{graphicx}
\usepackage{amssymb}
\usepackage{epstopdf}
\usepackage{amsmath}
\usepackage{color}
\usepackage{hyperref}
\usepackage{pdfsync}
\usepackage{float}
\usepackage{verbatim}
\usepackage{enumitem}
\usepackage{tikz}
\usepackage{tikz-cd}
\usetikzlibrary{arrows,decorations.pathmorphing,backgrounds,positioning,fit,petri}
\usetikzlibrary{arrows,patterns,matrix}
\usetikzlibrary{decorations.markings}
\usepackage[all,cmtip]{xy}

\definecolor{job}{RGB}{200,65,0}

\definecolor{sz}{RGB}{0,0,200}

\xyoption{matrix}
\xyoption{arrow}
\usepackage[table]{colortbl}
 
\tikzset{help lines/.style={step=#1cm,very thin, color=gray},
help lines/.default=.5} % draws a grid spaced #1 cm
\tikzset{thick grid/.style={step=#1cm,thick, color=gray},
thick grid/.default=1} % draws a grid spaced #1 cm

\usepackage{tcolorbox}  %to put box around text and math environment 
\tcbset{colframe=black,colback=white, width=(\linewidth-2mm), before=,after=\hfill,} 
\usepackage{graphicx}
\usepackage{tikz}
\usetikzlibrary{mindmap,shapes,snakes,matrix}
\tikzstyle{ann}=[fill=white, inner sep=1pt, font=\footnotesize{#1}]
\tikzstyle{annfar}=[inner sep=2pt, font=\footnotesize{#1}]
\tikzstyle{annfarer}=[inner sep=3pt, font=\footnotesize{#1}]
\tikzstyle{annrot}=[fill=white, text=blue!75!black, inner sep=1pt, font=\footnotesize{#1}]
\tikzstyle{wall}=[thick]
\tikzstyle{nullwall}=[thick, dotted]

\newtheorem{thm}{Theorem}[section]
\newtheorem{thm*}{Theorem}
\newtheorem{lem}[thm]{Lemma}
\newtheorem{cor}[thm]{Corollary}
\newtheorem{prop}[thm]{Proposition}

\theoremstyle{definition}
\newtheorem{defn}[thm]{Definition}

\newtheorem{exm}[thm]{Example}

\newtheorem{notation}[thm]{Notation}
\newtheorem{construction}[thm]{Construction}

\newtheorem{rmk}[thm]{Remark}
 
\numberwithin{equation}{section}

 \newcounter{tmp}

\DeclareMathOperator{\Coker}{Coker}%\newcommand{\coker}{\text{coker}}
\DeclareMathOperator{\Hom}{Hom}%
\DeclareMathOperator{\Ext}{Ext}%
\DeclareMathOperator{\End}{End}%
\DeclareMathOperator{\add}{add} % additive category generated by ...
\DeclareMathOperator{\pd}{pd}%
\DeclareMathOperator{\id}{id}%
\DeclareMathOperator{\modd}{\textendash mod}%
\DeclareMathOperator{\Modd}{\textendash Mod}%
\DeclareMathOperator{\Homeo}{Homeo}%

\newcommand{\field}[1]{\mathbb{#1}}
\newcommand{\ZZ}{\ensuremath{{\field{Z}}}}

\newcommand{\RR}{\ensuremath{{\field{R}}}}

\newcommand{\NN}{\ensuremath{{\field{N}}}}

\newcommand{\bS}{\ensuremath{{\field{S}}}}

\newcommand{\e}{\varepsilon}

\newcommand{\commentout}[1]{}

 \newcommand{\bR}{{\RR}}
 
\newcommand{\cA}{\ensuremath{{\mathcal{A}}}}

\newcommand{\cC}{\ensuremath{{\mathcal{C}}}}

\newcommand{\cM}{\ensuremath{{\mathcal{M}}}}

\newcommand{\cS}{\ensuremath{{\mathcal{S}}}}

\newcommand{\Top}{\operatorname{top}}

\newcommand{\Ker}{\operatorname{Ker}}

\newcommand{\Ima}{\operatorname{Im}}

\newcommand{\bk}{\mathbf{k}}
\newcommand{\pwf}{{\text{\textendash}\operatorname{pwf}}}
\newcommand{\bpwf}{{\text{\textendash}\operatorname{bpwf}}}
\newcommand{\spwf}{{\text{\textendash}\operatorname{spwf}}}

 \title{Continuous Nakayama Representations}
  \author{Job D.~Rock}
 \address{Ghent University, 9000 Ghent, East Flanders, Belgium}
 \email{job.rock@ugent.be}
 \author{Shijie Zhu}
\address{School of Science, Nantong University, Nantong, Jiangsu, P.R.China, 226019}
\email{shijiezhu@ntu.edu.cn}


 
\begin{document}
\begin{abstract}
    We introduce continuous analogues of Nakayama algebras.
    In particular, we introduce the notion of (pre-)Kupisch functions, which play a role as Kupisch series of Nakayama algebras, and view continuous Nakayama representations as a special type of representation of $\RR$ or $\bS^1$.
    We investigate equivalences and connectedness of the categories of Nakayama representations.
    Specifically, we prove that orientation-preserving homeomorphisms on $\RR$ and on $\bS^1$ induce equivalences between these categories. Connectedness is characterized by a special type of points called separation points determined by (pre-)Kupisch functions.
    We also construct an exact embedding  from the category of finite-dimensional representations for any  finite-dimensional Nakayama algebra, to a category of continuous Nakayama representaitons.
\end{abstract}

 \maketitle
 
\setcounter{tocdepth}{1}
\tableofcontents

\section{Introduction}
\subsection{Background}
A finite-dimensional algebra $A$ is Nakayama if it is both left and right serial. That is, left and right indecomposable  projective $A$-modules have unique composition series.
A basic Nakayama algebra is isomorphic to a quiver algebra $\bk Q/I$ where $Q$ is either an $\mathbb A_n$ type quiver with straight orientation, or an $\widetilde{\mathbb{A}}_n$ type quiver with cyclic orientation. 
Any basic Nakayama algebra can be determined by its Kupisch series $(l_1,l_2,\cdots, l_n)$, where $l_i$ encodes the length of the $i$-th indecomposable projective module.
In representation theory, Nakayama algebras have finite representation type and are considered as one of the most well-known classes of algebras.
Many homological properties of Nakayama algebras have been revealed. For instance, Gustafson \cite{Gu} showed that the global dimension of a basic Nakayama algebra with $n$ non-isomorphic simple modules are bounded by $2n-2$.
In \cite{Rin1}, Ringel characterized Gorenstein projective modules over Nakayama algebras. In \cite{Rin2} and \cite{S1}, the authors studied the finitistic dimension of Nakayama algebras.
In \cite{S2}, Sen characterizes Nakayama algebras which are higher Auslander. In \cite{MS}, the authors classifies quasi-hereditary Nakayama algebras.

As the recent development of persistence theory in topological data analysis, persistence modules has been extensively studied in both representation theory and data science.
The pointwise finite-dimensional $\RR$-representations, which appear as one-parameter persistence homology, are well understood classes of persistence modules.
Any pointwise finite-dimensional $\RR$-representation has a unique barcode decomposition as a direct sum of interval modules $M_U$ \cite{C,GR}.
Recently, Hanson and Rock \cite{HR} considered pointwise finite-dimensional $\bS^1$-representations and showed that any such representation can be uniquely decomposed as a direct sum of (possibly infinitely many) string modules $\overline{M}_U$ and finitely many Jordan cells.

\subsection{Continuous Nakayama Representations}
In this paper, we consider pointwise finite-dimensional (pwf) representations over $\RR$ or over $\bS^1$ subject to some relations given by a  (pre-)Kupisch function. 
These $\RR$- and $\bS^1$-representations are continuous analogues of representations of Nakayama algebras, which we call continuous Nakayama representations.

Consider $\RR$ as a category where the objects are points in $\RR$ and there is a unique morphism $g_{xy}:x\to y$ if $x\leq y$.
An $\RR$-representation $M$ over a field $\bk$ is a covariant functor $M:\RR\to\bk$-$\mathrm {Vec}$. For any interval $U$, an interval module $M_U$ is an $\RR$-representation such that $$M_U(x)=\begin{cases}\bk &x\in U \\ 0 & x\not\in U \end{cases} \text{\ and\ } M_U(x\to y)=\id_{\bk} \text{\ for\ } x\leq y\in U.$$

We similarly consider $\bS^1$ as a category where the objects are elements of $\bS^1$ and morphisms move counter-clockwise.
A pwf $\bS^1$-representation $M$ over a field $\bk$ is a covariant functor $M:\bS^1\to\bk$-$\mathrm{Vec}$. 
A string module $\overline{M}_U$ is the ``push-down'' of the interval module $M_U$ via the covering map $p:\RR\to\bS^1$.

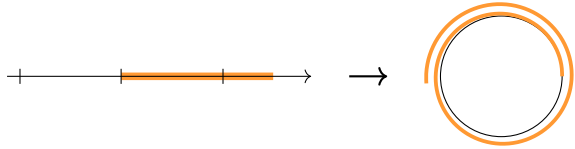
\begin{figure}[h]
    \centering
    \begin{tikzpicture}
    %orange bar
    \draw[line width = 1mm, orange!80!white] (-5,0) -- (-3,0);
    %real line
    \draw[->](-6.5,0)--(-2.5,0);
    %integer tics
    \foreach \x in {-6.33,-5,-3.66}
        \draw (\x,0.1) -- (\x,-0.1);
    %arrow
    \draw[thick, ->] (-2,0) -- (-1.5,0);
    %circle
    \draw (0.8,0) arc (0:360:0.8);
    %orange spiral
    \draw [domain=0:9.52, variable=\t, smooth, samples=75,line width = 0.5mm, orange!80!white] plot ({\t r}: {0.8+0.02*\t});
    \end{tikzpicture}
    \caption{Push-down of the interval $\RR$-module $M_{[0,1.5]}$ to the $\bS^1$-module $\overline M_{[0,1.5]}$.}\label{pushdown_fig}
\end{figure}

Continuous Nakayama representations are pwf representations on an interval $I\subseteq \RR$ or on $\bS^1$ that are compatible with some relations given by (pre-)Kupisch functions (Definitions~\ref{def:pre-Kupisch function}~and~\ref{defn:Kupisch}).
Similar to the Kupisch series of Nakayama algebras, a (pre-)Kupisch function $\kappa$ determines an indecomposable projective representation $M_{[t,t+\kappa(t)]}$ (or $\overline M_{[t,t+\kappa(t)]}$) at each point $t\in\RR$.
Thus, each indecomposable representation is $M_U$ (or $\overline M_U$) for some interval $U\subseteq [t,t+\kappa(t)]$.
We denote the full subcategory of pwf $\RR$-representations compatible with a pre-Kupisch function $\kappa$ on an interval $I$ by $(I,\kappa)\pwf$ and the full subcategory of pwf $\bS^1$-representations compatible with a Kupisch function $\kappa$ by $(\bS^1,\kappa)\pwf$.
We call them categories of {\bf continuous Nakayama representations}.

\subsection{Results}  
We investigate the connectedness and equivalences of categories $(I,\kappa)\pwf$ or $(\bS^1,\kappa)\pwf$ among various (pre-)Kupisch functions and their relation to discrete Nakayama representations. 

Recall that an orientation-preserving homeomorphism between intervals $I$ and $J$ is an increasing bijective map $f:I\to J$.
Given a pre-Kupisch function $\kappa$ on an interval $I$, an orientation preserving homeomorphism $f:I\to J$ induces a push-forward pre-Kupisch function $f_*\kappa$ on $J$ (Definition \ref{def:push-forward}).
In fact, this gives rise to an equivalence of categories $(I,\kappa)\pwf\to(J,f_*\kappa)\pwf$ that sends interval modules $M_U$ to $M_{f(U)}$ (Theorem \ref{interval equiv}).
Similarly, an orientation-preserving homeomorphism on $\bS^1$ is a degree $1$ circle homeomorphism $f:\bS^1\to\bS^1$ (\cite{H}).
For a Kupisch function $\kappa$, such an $f$ induces a push-forward Kupisch function $f_*\kappa$ (Definition~\ref{def:circle push-forward}). This also gives an equivalence $(\bS^1,\kappa)\pwf\to(\bS^1, f_*\kappa)\pwf$ (Theorem \ref{iso1}). 

The converse of Theorems~\ref{interval equiv}~and~\ref{iso1} may not hold in case the category $(I,\kappa)\pwf$ (or $(\bS^1,\kappa)\pwf$) is disconnected.
This is because one can construct an equivalence by permuting the orthogonal components of the representation category (Example \ref{n piece}).

% We describe the connectedness of the categories of continuous Nakayama representations using an analytical property of the (pre-)Kupisch function.
% Given a pre-Kupisch function $\kappa$ on $I$, we define a separation point to be a special kind of point in $I\setminus\partial I$ (Definition \ref{def:separation pt}).
% Denote by $\cS(\kappa)$ the set of all separation points.
% Our first result concerns connectedness of $(I,\kappa)\pwf$ (or $(\bS^1,\kappa)\pwf$) for a (pre-)Kupisch function $\kappa$.

In order to describe the connectedness of the categories of continuous Nakayama representations, we introduce the notion of separation points, using an analytical property of (pre-)Kupisch functions (Definition \ref{def:separation pt}). Denote by $\cS(\kappa)$ the set of all separation points of $\kappa$. We show the following result concerning the connectedness of categories $(I,\kappa)\pwf$ (or $(\bS^1,\kappa)\pwf$).

  \begingroup
\setcounter{tmp}{\value{thm}}% store current value of theorem counter
\setcounter{thm}{0} %assign desired value to theorem counter
\renewcommand\thethm{\Alph{thm}}% locally redefine the representation of the theorem counter

\begin{thm} \label{main thm A}
(1) Let $\kappa$ be a pre-Kupisch function on an interval $I$.
Then $(I,\kappa)\pwf$ is a connected additive category if and only if $\cS(\kappa)=\emptyset$.

(2) Let $\kappa$ be a Kupisch function.
Then $(\bS^1,\kappa)\pwf$ is a connected additive category if and only if $|\cS(\kappa)\cap[0,1)|\leq 1$.
\end{thm}

When (pre-)Kupisch functions $\kappa$ and $\lambda$ have no separation points, any equivalence $(I,\kappa)\pwf\cong(J,\lambda)\pwf$ (or $(\bS^1,\kappa)\pwf\cong(\bS^1,\lambda)\pwf$) is induced by a push-forward $\lambda=f_*\kappa$ via some orientation preserving homeomorphism $f$.
  
\begin{thm}\label{main thm B}
(1) Let $F:(I,\kappa)\pwf\to(J,\lambda)\pwf$ be an equivalence of categories where $\kappa$ and $\lambda$ are pre-Kupisch functions on intervals $I$ and $J$ respectively, such that $\cS(\kappa)=\cS(\lambda)=\emptyset$.
Then $F$ induces an orientation preserving homeomorphism $f:I\to J$ such that $\lambda=f_* \kappa$.

(2) Let $F:(\bS^1,\kappa)\pwf\to(\bS^1,\lambda)\pwf$ be an equivalence of categories, where $\kappa$ and $\lambda$ are Kupisch functions such that $\cS(\kappa)=\cS(\lambda)=\emptyset$.
Then $F$ induces an orientation preserving  homeomorphism $f:\bS^1\to\bS^1$ such that $\lambda=f_* \kappa$.
\end{thm}

To reveal the relation with representations of discrete Nakayama algebras, we associate to each basic Nakayama algebra $A$ having Kupisch series $(l_0,l_1,\cdots, l_{n-1})$ an associated Kupisch function $\kappa_A$ (Definition~\ref{associate kupish}). 
% $$
% \kappa_A(t)=  \frac{i+l_{i}}{n}-t+k,\ \ \frac{i}{n}+k\leq t< \frac{i+1}{n}+k, \text{for $0\leq i<n$}.
% $$
Then we show the following result.
\begin{thm}\label{main thm C}
Let $A$ be a basic connected Nakayama algebra with a Kupisch series $(l_0,l_1,\cdots, l_{n-1})$ and $\kappa_A$ its associated Kupisch function.
Then there is an exact embedding $F: A\modd\to (\bS^1, \kappa_A)\pwf$, which preserves projective objects.
Moreover, if $A$ is a linear Nakayama algebra, then there is an exact embedding $L:A\modd\to (\RR, \kappa_A)\pwf$, which preserves projective objects.
\end{thm}

\endgroup

\subsection{Further Inverstigations}
It is also worth mentioning a connection of our results with dynamical systems.  Pre-Kupisch (respectively Kupisch) functions naturally give rise to self-maps on $\RR$  (respectively $\bS^1$) and hence defines a dynamical system on $\RR$ (respectively $\bS^1$). The construction of push-forward of (pre-)Kupisch functions corresponds to the notion of topological conjugacy in dynamic systems. Therefore Theorem~\ref{main thm B} asserts that the classification of categories $(\bS^1,\kappa)\pwf$ and $(\RR,\kappa)\pwf$ is equivalent to classifying their corresponding dynamical system up to positive topologically conjugacy.

\subsection*{Acknowledgements}
JR is supported at Ghent University by BOF grant 01P12621.
%travel grant BOF.PDO.2022.0005.01
JR would like to thank Karin M.~Jacobsen, Charles Paquette, and Emine Y{\i}ld{\i}r{\i}m for helpful dicussions.
SZ would like to thank Shrey Sanadhya for helpful discussions.

\section{$\RR$- and $\bS^1$-representations}\label{sec:persistence modules}
%\section{Persistence modules}\label{sec:persistence modules}

In this section we recall representations (persistence modules) over $\RR$ and over $\bS^1$.
We describe interval modules for $\RR$, string modules for $\bS^1$, and the relationship between them via covering theory and orbit categories.

\subsection{$\RR$-representations}\label{sec:RR persistence}

Denote by $\RR$ the category of real numbers, where objects are real numbers and there is a unique generating morphism $g_{xy}:x\to y$ if $x\leq y$ and $g_{xx}=id_x$.
Composition is given by $g_{yz}\circ g_{xy}=g_{xz}$.
It follows that 
\begin{displaymath}
    \Hom_{\RR}(x,y) = \begin{cases}
        \{g_{xy}\} & x\leq y \\
        \emptyset &x>y
    \end{cases}.
\end{displaymath}

An \textbf{$\RR$-representation} over a field $\bk$ is a covariant functor $M: \RR \to \bk$-$\mathrm {Vec}.$ 
A morphism of $\RR$-representations is a natural transformation $f: M\to N$.
That is, a morphism is a collection of $\bk$-linear maps $f(x)$ for each $x\in\RR$ such that the following diagram commutes:
\begin{displaymath}
    \begin{tikzcd}
        M(x) \arrow[d,"f(x)",swap] \arrow[r,"M(g_{xy})"] & M(y) \arrow[d,"f(y)"] \\
        N(x) \arrow [r, "N(g_{xy})",swap] & N(y).
    \end{tikzcd}
\end{displaymath}

Let $(a,b]\subseteq\RR$ be an interval.
The $\RR$-representation $M_{(a,b]}$ is defined as:
\begin{align*}
    M_{(a,b]}(t) &= \begin{cases}
        \bk & a<t\leq b\\
        0 & \text{otherwise}
    \end{cases}
    &
    M_{(a,b]}(g_{st}) &= \begin{cases}
        1 & a < s\leq t\leq b \\
        0 & \text{otherwise}.
    \end{cases}
\end{align*}
In general, if $U$ is any interval, we define the {\bf interval module} $M_U$ similarly.

\begin{exm}\label{exm:morphism of RR persistence modules}
    There is a nonzero homomorphism $f: M_{[1,3]}\to M_{[0,2]}$ with 
    \begin{displaymath}
        f(s)=\begin{cases} 
            0=M_{[1,3]}(s)\stackrel{0}\to M_{[0,2]}(s)={\bk}& 0\leq s<1\\
            {\bk}=M_{[1,3]}(s)\stackrel{=}\to M_{[0,2]}(s)={\bk} & 1\leq s\leq 2 \\ 
            {\bk}=M_{[1,3]}(s)\stackrel{0}\to M_{[0,2]}(s)=0 & 2< s\leq 3
        \end{cases}
    \end{displaymath}
    We see that $\Ima f=M_{[1,2]}$, $\Ker f=M_{(2,3]}$ and $\Coker f=M_{[0,1)}$.
\end{exm}

We call a module $M$ \textbf{pointwise finite-dimensional} (\textbf{pwf}) if $\dim M(x)<\infty$, for each $x\in\RR$.
Denote by $\RR\pwf$ the category of pwf $\RR$-representations. 
Indecomposable objects in $\RR\pwf$ are exactly the interval modules.
 
Recall that an abelian category is \textbf{Krull--Remak--Schmidt} if every object decomposes uniquely up to isomorphism into a (possibly infinite) direct sum of indecomposables and each indecomposable has a local endomorphism ring.
 
\begin{thm}\cite{C,GR}\label{barcode decom}
    Any pointwise finite-dimensional  $\RR$-representation decomposes uniquely up to isomorphism into a direct sum of interval modules.
    In particular, $\RR\pwf$ is Krull--Remak--Schmidt.
\end{thm} 

For convenience, we introduce the following definition and lemma.
\begin{defn}\label{def:left intersect}
    Define the \textbf{left intersection} of intervals $U$ and $V$ as:  
    \begin{displaymath}
        U\cap_L V = \begin{cases}
            U\cap V & \text{if } (\forall x\in V\setminus U,\,  \forall y\in U,\,  y< x) \text{ and } (\forall x\in V,\, \forall y\in U\setminus V,\,  y< x) \\
            \emptyset & \text{otherwise}.
        \end{cases}
    \end{displaymath}
\end{defn}
\begin{lem}\label{lem:left lem}
    Let $M_U$ and $M_V$ be interval modules in $\RR\pwf$.
    Then
    \begin{displaymath}
        \Hom_{\mathbb R}(M_V,M_U)=\begin{cases}
            \bk & U\cap_L V \neq \emptyset \\
            0 & U\cap_L V = \emptyset.
        \end{cases}
    \end{displaymath}
\end{lem}

\subsection{$\bS^1$-representations}\label{sec:bS persistence}

In this section we follow \cite{HR} in our definition of $\bS^1$.
 
Denote by $\bS^1$ the category of points $e^{2i\pi\theta}$ for $\theta\in\RR$.
There is a unique generating morphism $g_{xy}:x\to y$ for all pairs $x,y\in \bS^1$ where $g_{xx} = id_x$.
Additionally, we have a unique map $\omega_x:x\to x$ which captures going around $\bS^1$ exactly once from $x$ to $x$.
Consider points $x=e^{2i\pi\theta}$, $y=e^{2i\pi\phi}$, and $z=e^{2i\pi\psi}$, where $\theta\leq\phi\leq \psi$, $\phi-\theta< 1$, and $\psi-\phi< 1$.
Composition $g_{yz}g_{xy}$ is defined as 
\begin{displaymath}
    g_{yz}g_{xy}=\begin{cases}
        g_{xz} & 0\leq \psi-\theta < 1 \\
        g_{xz}\circ \omega_x=\omega_z \circ g_{xz} & 1\leq \psi - \theta < 2.
    \end{cases}
\end{displaymath}
Define $\omega^0_x := id_x = g_{xx}$.
In particular, for any $y\neq x$ in $\bS^1$, we have $\omega_x=g_{yx}\circ g_{xy}$.
It follows that
\begin{displaymath}
    \Hom_{\bS^1}(x,y) = \left\{ g_{xy}\circ \omega^n_x = \omega^n_y\circ g_{xy} : n\in{\mathbb N} \right\}.
\end{displaymath}
An \textbf{$\bS^1$-representation} over a field $\bk$ is a covariant functor $M: \bS^1 \to \bk\text{-}Vec.$ 
An $\bS^1$-representation homomorphism is a natural transformation $f:M\to N$.

In fact, covering theory (see \cite{A},\cite{BG},\cite{G},\cite{GR}) provides a convenient tool to understand the $\bS^1$-representations. 
Recall that a \textbf{spectroid} is  a Hom-finite (but not additive) $\bk$-category in
which all objects have local endomorphism algebras, and distinct objects are not
isomorphic (see \cite{GR}). Since the abelian group $\mathbb Z$ can act freely on the spectroid $\RR$ by $n\cdot x=x+n$, 
The category $\bS^1$ can be considered as the orbit category $\RR/\mathbb Z$ via the Galois covering
\begin{align*}
    p:\RR & \to \bS^1 \cong \RR/\mathbb Z \\
    t & \mapsto t \mod 1.
\end{align*}

Denote by $\RR \Modd$ and $\bS^1 \Modd$ the category of all representations on $\RR$ and $\bS^1$, respectively.
 From covering theory, the Galois covering $p$ gives rise to two important functors between these two categories:
 (1) the pull-up functor $p^\bullet: \bS^1 \Modd\to \RR \Modd$ sending $M\mapsto M\circ p$ and
 (2) the push-down functor $p_\bullet:\RR \Modd \to \bS^1 \Modd$ which satisfies $p_\bullet M(x)=\oplus_{t\in p^{-1}(x)} M(t)$.
 It is well-known that $(p_\bullet, p^\bullet)$ forms an adjoint pair and $p_\bullet$ preserves finitely generated objects.
 However, we also remind the readers that $p_\bullet$ does not preserve pointwise finite-dimensional modules.
 (For example, consider $p_\bullet ( \bigoplus_{i\in\ZZ} M_{[i,i]})$.)
 In the following, we give an explicit description of indecomposable $\bS^1$-representations using the push-down functor.

Define an $\bS^1$ string module $\overline{M}_U$ as the push-down of an
interval $\RR$-representation $M_U$ for some bounded interval $U$. 
Explicitly,
for each $x\in\bS^1$,  denote by $p_U^{-1}(x):=p^{-1}(x)\cap U$. Then $\overline{M}_U(x) :=\bigoplus\limits_{b\in p_U^{-1}(x)}M_U(b)$. 
  
Denote the elements of $p_U^{-1}(x)$ by $\{b_{i,x}\}_{i=1}^{|p_U^{-1}(x)|}$ such that $b_{i,x}< b_{i+1,x}$ for all $i$. 
For each generating morphism $g_{xy}:x\to y$, where $\overline{M}_U(x)\neq 0$, the linear maps $\overline{M}_U(g_{xy})$ are defined on the basis $\{b_{i,x}\}_{i=1}^{|p_U^{-1}(x)|}$.
\begin{displaymath}
    \overline{M}_U(g_{xy})(b_{i,x}) = \begin{cases}
        b_{j,y} & \exists b_{j,y}\in p_U^{-1}(y) \text{ such that } 0\leq b_{j,y}-b_{i,x} < 1 \\ 
        0 & \text{otherwise}.
    \end{cases}
\end{displaymath}
Due to $\overline{M}_U(\omega_x)=\overline{M}_U(g_{yx}g_{xy})$,
\begin{displaymath}
    \overline{M}_U(\omega_x)(b_{i,x}) = \begin{cases}
        b_{i+1,x} &  1\leq i <|p_U^{-1}(x)| \\ 
        0 & i=|p_U^{-1}(x)|.
    \end{cases}
\end{displaymath}
 
If $f: M_U\to M_V$ is a morphism of $\RR$-representations, where $U$ and $V$ are bounded intervals, then push-down functor induces a morphism $\bar f: \overline{M}_U\to \overline{M}_V$ as $\bar f(x)=\oplus f(b_{i,x}): \oplus M_U(b_{i,x})\to \oplus M_V(b'_{j,x})$, where $b_{i,x}\in p_U^{-1}(x)$ and $b'_{j,x}\in p_V^{-1}(x)$.

\begin{thm}\cite{HR}\label{decom s1}
    Any pointwise finite-dimensional $\bS^1$-representation decomposes into a direct sum of string modules and finitely-many Jordan cells.
    In particular, $\bS^1\pwf$ is Krull--Remak--Schmidt.
\end{thm}
As we are constructing a continuous analogue of Nakayama representions, we study $\bS^1$-representations compatible with Kupisch functions. Thus, we are only interested in string modules.
   
Let $U$ be an interval of $\RR$.
For any $i\in\ZZ$, denote by $U+i$ the interval $\{x+i \mid x\in U\}$. 
The $\RR$-representations $M_{U+i}$ are called translated modules of $M_U$. Denote by $\widetilde{X}= X\circ p$ the pull-up of $\bS^1$-representation $X$. One can check that $\widetilde{\overline{M}}_U = \overline{M}_U\circ p = \bigoplus\limits_{i\in\mathbb Z}M_{U+i}$.
We have the following Lemma which should be compared with Lemma \ref{lem:left lem}.   
\begin{lem}\label{lem:n left lem}
    Let $\overline{M}_U$ and $\overline{M}_V$ be $\bS^1$ string modules.
    Then
    \begin{displaymath}
        \Hom_{\bS^1}(\overline{M}_U, \overline{M}_V) \cong \bigoplus\limits_{i\in\mathbb Z}\Hom_\RR(M_U,M_{V+i}).
    \end{displaymath}
\end{lem}
\begin{proof}
    Since the push-down is left adjoint to the pull-up, we have
    \begin{displaymath}
        \Hom_{\bS^1}(\overline{M}_U, \overline{M}_V) \cong \Hom_\RR(M_U,\widetilde{\overline{M}}_V) = \Hom_\RR(M_U,\bigoplus\limits_{i\in\mathbb Z}M_{V+i}) \cong \bigoplus\limits_{i\in\mathbb Z}\Hom_\RR(M_U,M_{V+i}).\qedhere
    \end{displaymath}
\end{proof}
\begin{cor}[to Lemma \ref{lem:n left lem}]\label{cor:S1 hom space}
    Let $\overline{M}_U$ and $\overline{M}_V$ be $\bS^1$ string modules.
    Then
   \begin{displaymath}
        \dim_{\bk}\Hom_{\bS^1} (\overline{M}_U,\overline{M}_V) =
        \left|\{i: U\cap_L (V+i)\neq \emptyset\}\right|.
    \end{displaymath}
\end{cor}
\begin{proof}
    Since there are only finitely many $i$ such that $\Hom_\RR(M_U, M_{V+i})\neq 0$ for bounded intervals $U$ and $V$, this is always a finite direct sum.  
    Hence, the statement follows from Lemmas \ref{lem:left lem} and \ref{lem:n left lem}.
\end{proof}
 
 Recall that a module $M$ is called a {\bf brick} if $\End(M)$ is a division ring. In particular, bricks are indecomposable. 
 \begin{cor}[to Corollary~\ref{cor:S1 hom space}] \label{cor:S1 brick}
 An $\bS^1$ string module $\overline M_U$ is a brick if and only if $|x-y|<1$ for $\forall x,y\in U$.
\end{cor}

The push-down functor does not preserve pwf representations, so we instead use the concept of orbit categories and restrict the push-down functor to a subcategory of $\RR\pwf$.

\begin{defn}\label{def:bpwf and spwf}
Let $\RR\bpwf$ be the full subcategory of $\RR\pwf$, where for each object $M$ in $\RR\bpwf$, there exists $a\leq b\in \RR$ such that $M(x)=0$ if $x\leq a$ or $x\geq b$.

Let $\bS^1\spwf$ be the full subcategory of $\bS^1\pwf$ whose objects do not have a band summand.
\end{defn}

It is a straightforward check that $\RR\bpwf$ is a Serre subcategory of $\RR\pwf$ and is thus abelian.
Moreover, $\RR\bpwf$ is Krull--Remak--Schmidt.
Since we have the cyclic orientation on $\bS^1$, the category $\bS^1\spwf$ is a Serre subcategory of $\bS^1\pwf$; in particular it is also abelian and Krull--Remak--Schmidt.

By Theorem~\ref{decom s1}, we see that every pwf $\bS^1$-representation that does not have a band summand is isomorphic to the push-down of some object $M$ in $\RR\bpwf$.
By restricting the push-down functor to $\RR\bpwf$, we view $\bS^1\spwf$ as the orbit category of $\RR\bpwf$ where the autoequivalence on $\RR\bpwf$ is induced by the action of $\ZZ$.

\section{(Pre-)Kupisch Functions and Compatible Representations}\label{sec:Kupisch functions and compatible representations}

In this section we define a continuous analogue to Kupisch series called (pre-)Kupisch functions.
Compatible modules form a subcategory of $\RR\pwf$ (or $\bS^1\pwf$) that behaves like the category of finite-dimensional representations of a finite-dimensional Nakayama algebra.

Throughout this paper, functors between $\RR\pwf$, $\bS^1\pwf$, and their subcategories are assumed to be $\bk$-linear and additive.
 
\subsection{Pre-Kupisch functions and compatible $\RR$-representations}\label{sec:pre-Kupisch functions}
 
\begin{defn}\label{def:pre-Kupisch function}\label{def:compatible with pre-Kupisch function}
    Let $I\subseteq \RR$ be an interval. 
    \begin{enumerate}
    \item A \textbf{pre-Kupisch function} on $I$ is $\kappa:I\to \RR^{>0}$ such that
        \begin{itemize}
            \item $K(t):=\kappa(t)+t$ is increasing on $I$ and
            \item For all $t\in I$, we have $[t,K(t)]=[t,\kappa(t)+t]\subseteq I$.
        \end{itemize}
    \item We call an $\RR$-representation \textbf{compatible with $\kappa$} if each of its indecomposable summands $M_U$ satisfies $ U\subseteq[t, K(t)]\subseteq I$, for some $t\in I$.
    Let $(I,\kappa)\pwf$ be the full subcategory of $\RR\pwf$ consisting of precisely the modules compatible with $\kappa$.
    \end{enumerate}
\end{defn}

\begin{rmk}\label{rmk:pre-Kupisch function}
We remark on some immediate facts about pre-Kupisch functions.
\begin{enumerate}
    \item For any pre-Kupisch function $\kappa$ on an interval $I$, the category $(I,\kappa)\pwf$ is a Serre subcategory of $\RR\pwf$ and is thus abelian and Krull--Remak--Schmidt.
    \item For a pre-Kupisch function $\kappa:\RR\to\RR^{>0}$, the restriction $\kappa|_I:I\to \RR^{>0}$ on an interval $I$ is not always a pre-Kupisch function. In particular, there are no pre-Kupisch functions on intervals of the form $[a,b]$, $(a,b]$ or $(-\infty,b]$.
    This is because $[b,b+\kappa(b)]\not\subseteq I$. 
    
    \item  For any pre-Kupisch function $\kappa:\RR\to\RR^{>0}$, the restriction $\kappa|_{\bR^{\geq 0}}:\bR^{\geq 0}\to \RR^{>0}$ is a pre-Kupisch function on $\bR^{\geq 0}$.
    
    \item A pre-Kupisch function $\kappa$ determines the indecomposable projective modules in $(I,\kappa)\pwf$, which are the interval modules $M_{[t,K(t)]}$ and $M_{(t,K(t)]}$, for all $t\in I$.
\end{enumerate}
\end{rmk}

We make use of the following fact.
\begin{prop}\label{prop:interval bijection}
Let $I,J$ be intervals. An increasing bijection $f:I\to J$ is continuous and hence a homeomorphism.  \end{prop}

Increasing bijections $f:I\to J$ are also called  orientation-preserving homeomorphisms.
Denote by $\Homeo_+(I,J)$ the set of orientation-preserving homeomorphisms from $I$ to $J$ and $\Homeo_+(I)$ the set of orientation-preserving homeomorphisms from $I$ to $I$.
The following result shows that an orientation-preserving homeomorphism preserves subintervals and their left intersections.
\begin{lem}\label{lem:left intersection lemma}
    Let $I,J$ be intervals in $\RR$, $f\in\Homeo_+(I,J)$, and $U, V$ intervals in $I$.
    Then $f(U \cap_L V)=f(U) \cap_L f(V)$.
\end{lem}
\begin{proof} First, we claim that $U\cap_L V=\emptyset$ if and only if $f(U)\cap_L f(V)=\emptyset$. Indeed,  if $U\cap_L V\neq \emptyset$, then for   $(\forall x\in V\setminus U,\,  \forall y\in U,\,  y< x)$ and $(\forall x\in V,\, \forall y\in U\setminus V,\,  y< x)$. Since $f$ is an increasing bijection, $\forall x\in f(V)\setminus f(U),\,  \forall y\in f(U)$, it follows that  $f^{-1}(x)\in V\setminus U$ and $f^{-1}(y)\in U$, hence  $f^{-1}(y)< f^{-1}(x)$, which implies $y<x$. Similarly, $(\forall x\in f(V),\, \forall y\in f(U)\setminus f(V),\,  y< x)$. So $f(U)\cap_L f(V)\neq \emptyset$. Then the converse direction of the claim follows directly from the fact that $f^{-1}:J\to I$ is again an increasing homeomorphism.

To prove the lemma, notice that if $U\cap_L V=\emptyset$, then 
$f(U \cap_L V)=\emptyset=f(U) \cap_L f(V)$. If $U\cap_L V\neq \emptyset$, then $U\cap_L V=U\cap V$. So $f(U\cap_L V)=f(U\cap V)=f(U)\cap f(V)=f(U)\cap_L f(V)$.\qedhere 

\end{proof}

Given a pre-Kupisch function $\kappa$ on an interval $I$ and a increasing bijection $I\to J$, we now define the push-forward $f_*\kappa$ and show that this is a pre-Kupisch function on $J$.

\begin{defn}\label{def:push-forward}
    Let $\kappa$ be a pre-Kupisch function on interval $I$ and let $f:I\to J$ be an increasing bijection.
    We define the \textbf{push-forward} by
    \begin{displaymath}
        f_*\kappa(t):= f(\kappa\circ f^{-1}(t)+f^{-1}(t))-t.
    \end{displaymath}
\end{defn}

Notice that by definition $f$ sends any interval $[t,t+\kappa(t)]$ to the interval $[f(t),f(t)+f_*\kappa(f(t))]$.

\begin{lem}\label{lem:pushforward pre} 
    If $\kappa$ is a pre-Kupisch function on an interval $I$ and $f\in\Homeo_+(I,J)$, then $f_*\kappa$ is a pre-Kupisch function on $J$.
    Furthermore, $f$ defines a bijection from the set of pre-Kupisch functions on $I$ to the set of pre-Kupisch functions on $J$.
\end{lem}
 
\begin{thm} \label{interval equiv}
Let $I, J$ be intervals of $\RR$. If $\kappa$ is a pre-Kupisch function on $I$ and $f\in\Homeo_+(I,J)$ then $(I,\kappa)\pwf\cong(J, f_*\kappa)\pwf$.
\end{thm}

\begin{proof}
An interval $U\subseteq [t,K(t)]\subseteq I$ if and only if $f(U)\subseteq [f(t),f(\kappa(t)+t)]=[f(t),f(t)+f_*\kappa(f(t))]\subseteq J$. Hence we construct a functor $F:(I,\kappa)\pwf\to (J, f_*\kappa)\pwf$ which sends interval modules to interval modules: $F(M_U)=M_{f(U)}$.
A nonzero morphism $\varphi: M_U\to M_V$ is determined by a scalar $c$ such that $\varphi(x)=c\cdot 1_\bk:M_U(x)\to M_V(x)$, for $\forall x\in V\cap_L U$.
Thus, we define $F(\varphi):M_{f(U)}\to M_{f(V)}$ as a morphism such that $F(\varphi)(x)=c\cdot 1_\bk$ for all $x\in f(V)\cap_L f(U)$.
Notice that $F$ is a dense functor.

By Lemmas \ref{lem:left lem} and \ref{lem:left intersection lemma}, we have an induced bijection
\begin{align*}
    \Hom_{(I,\kappa)\pwf}(M_U,M_V)\ &\cong \Hom_{(J, f_*\kappa)\pwf}(M_{f(U)},M_{f(V)})\\
    \varphi\ &\mapsto F(\varphi).
\end{align*} as $\bk$-vector spaces and so
$F:\Hom(M_{U},M_{V})\to \Hom(M_{f(U)},M_{f(V)})$, $\varphi\mapsto F(\varphi)$, is an isomorphism.
Since all modules are a direct sum of interval modules, extending additively, $F:(I, \kappa)\pwf\to (J, f_*\kappa)\pwf$ is an equivalence.
\end{proof}

\begin{cor}[to Theorem~\ref{interval equiv}] \label{R equiv 1}
Let $I$ be an interval with  a  pre-Kupisch function $\kappa$.
Then $$(I,\kappa)\pwf\cong \begin{cases} (\RR^{\geq 0},\lambda)\pwf, &I=[a,b) \text{\ or\ } [a,+\infty) \\ (\RR,\mu)\pwf, & I \text{\ is\ open}\end{cases}$$  
for some pre-Kupisch function $\lambda$ on $\RR^{\geq0}$ or $\mu$ on $\RR$.
\end{cor}

 We warn the reader that categories of the forms $(\RR,\mu)\pwf $ and $(\bR^{\geq 0},\lambda)\pwf$ are possibly equivalent, when they are not connected (Example \ref{N=Z}).

\subsection{Kupisch functions and compatible $\bS^1$-representations }
\begin{defn} \label{defn:Kupisch}{~}
\begin{enumerate}
    \item A pre-Kupisch function $\kappa:\RR\to \RR^{>0}$ on $\RR$ is called a \textbf{Kupisch function} if $\kappa(t+1)=\kappa(t)$ for all $t\in\RR$.
    \item We call a pwf{~}$\bS^1$-representation \textbf{compatible with a Kupisch function} $\kappa$ if all of its indecomposable summands are strings and each of its indecomposable summands $\overline{M}_U$ satisfies $U\subseteq[t,K(t)]=[t,t+\kappa(t)]$ for some $t\in \RR$.
Let $(\bS^1,\kappa)\pwf$ be the full subcategory of $\bS^1\pwf$ consisting representations compatible with $\kappa$.
\end{enumerate}
\end{defn}

\begin{rmk}\label{rmk:S kappa is serre subcategory}
    We remark on two immediate consequences of Definition~\ref{defn:Kupisch}.
    \begin{enumerate}
        \item For any Kupisch function $\kappa$, the category $(\bS^1,\kappa)\pwf$ is a Serre subcategory of $\bS^1\pwf$ and is thus abelian and Krull--Remak--Schmidt.
        Furthermore, $(\bS^1,\kappa)\pwf$ is a Serre subcategory of $\bS^1\spwf$.
        \item For any periodic pre-Kupisch function $\kappa$ with periodicity $r>0$ there is a Kupisch function $\lambda$ such that $(\RR,\kappa)\pwf\cong (\RR,\lambda)\pwf$.
        Define $\lambda=f_*\kappa$ where $f\in\Homeo_+(\RR)$ is given by $f(t)= \frac{t}{r}$ and use Theorem~\ref{interval equiv}.
        If $\kappa$ is a Kupisch function then $\lambda=\kappa$.
    \end{enumerate}
\end{rmk}

One should consider the Kupisch function $\kappa$ as the lift of a function $\bS^1\to \RR^{>0}$, which assigns the length of a projective string module at each point on the circle.

Given a Kupisch function $\kappa$, we construct new Kupisch functions using orientation-preserving homeomorphisms $f:\bS^1\to\bS^1$.

Recall that a homeomorphism $f:\bS^1\to\bS^1$ is called {\bf orientation-preserving} if its lift $\widetilde{f}$ through the universal covering $p: \RR\to \bS^1$, $p(t)=e^{2\pi it}$ is strictly increasing. An orientation preserving homeomorphism has degree $1$, namely $\widetilde{f}(t+1)-\widetilde{f}(t)=1$ for all $t$ (\cite{H}).
Denote by $\Homeo_+(\bS^1)$ the set of all the orientation-preserving homeomorphisms on $\bS^1$.
Notice that an orientation-preserving map need not be a simple rotation,
for example consider $f:\RR/\mathbb Z\to \RR/\mathbb Z$ given by $f(x) = x + b\sin(2\pi x)$ for any $0<b<\frac{1}{2\pi}$.

\begin{defn}\label{def:circle push-forward}
Let $f\in\Homeo_+(\bS^1)$ and choose a lift $\widetilde{f}$ of $f$. For any Kupisch function $\kappa$ we define the push-forward by $$(f_*\kappa)(t):= (\tilde f_*\kappa)(t)=\widetilde f(\kappa\circ\widetilde f^{-1}(t)+\widetilde f^{-1}(t))-t.$$
\end{defn}

Straightforward computations show the right side is independent of the choice of $\widetilde f$ and it is again a Kupisch function.
We have the following Lemma that follows from Lemma~\ref{lem:pushforward pre}.

\begin{lem}\label{lem:pushfoward K}
    If $\kappa$ is a Kupisch function and $f\in\Homeo_+(\bS^1)$ then $f_*\kappa$ is a Kupisch function.
    Furthermore, $f$ defines a bijection on the set of Kupisch functions.
\end{lem}

\begin{construction}\label{construction:push-down}
Let $\kappa$ and $\lambda$ be Kupisch functions.
Consider the subcategory $(\RR,\kappa)\bpwf$ and $(\RR,\lambda)\bpwf$ of $(\RR,\kappa)\pwf$ and $(\RR,\lambda)\pwf$, respectively.
Recall from Remark~\ref{rmk:pre-Kupisch function}(1) that these are abelian, Krull--Remak--Schmidt subcategories.

Recall that given an $\RR$-module $M$, the {\bf translated module} $M^{(i)}$ (under the $\mathbb Z$ action) is an $\RR$-representation such that $M^{(i)}(x)=M(x-i)$. Hence for an interval module $M_U$, its translated module $M^{(i)}_U=M_{U+i}$. Since a Kupisch function $\kappa$ is $1$-periodic, the translation $(-)^{(i)}:(\RR,\kappa)\bpwf\to (\RR,\kappa)\bpwf$ is an isomorphism with inverse $(-)^{(-i)}$. % (or $^iM(x)=M(x-i)$).
Let $F:(\RR,\kappa)\bpwf\to (\RR,\lambda)\bpwf$ be an additive covariant functor which preserves indecomposable objects and satisfies $F\circ (-)^{(1)}=(-)^{(1)}\circ F$.

We now define a functor $\overline{F}: (\bS^1,\kappa)\pwf \to (\bS^1,\lambda)\pwf$.
Recall Remark~\ref{rmk:S kappa is serre subcategory}(1) that these are abelian, Krull--Remak--Schmidt subcategories.
On indecomposables, define $\overline{F}(\overline{M}_U) := \overline{F(M_U)}$.
This is well-defined since $F(M^{(1)}_U)=(FM_U)^{(1)}$.

Given a homomorphism $f\in\Hom_{\bS^1,\kappa}(\overline M_U,\overline M_V)$, we identify $f$ as $(f_i:M_U\to M_{V+i})_{i\in\mathbb Z}$ due to Lemma~\ref{lem:n left lem}. Then $(F(f_i):FM_U\to F(M_{V}^{(i)})=(FM_V)^{(i)})_{i\in\mathbb Z}$ defines a homomorphism $\overline F(f)\in\Hom_{\bS^1,\lambda}(\overline {FM_U},\overline {FM_V})$.

To show $\overline F$ is a well-defined functor, it remains to show that $\overline F$ respects compositions. Let $f=(f_i):\overline M_U\to \overline M_V$ amd $g=(g_j): \overline M_V\to \overline M_W$ be morphisms in $(\bS^1,\kappa)\pwf$. Then the composition $g\circ f=(h_k):\overline M_U\to \overline M_V$ is given by $h_k=\sum_{d\in\mathbb Z}g_{k-d}^{(d)}f_d:M_U\to M_{W+k}$, where $g^{(d)}_j:M^{(d)}_V\to M^{(d)}_{W+j}$ are the translated morphisms.

Now $\overline F(g\circ f)= (F(h_k))_{k\in\mathbb Z} $, where the $k$-th component $F(h_k)=F(\sum_{d\in\mathbb Z}g_{k-d}^{(d)}f_d)=\sum_{d\in\mathbb Z}F(g_{k-d}^{(d)})\circ F(f_d)=\sum_{d\in\mathbb Z}F(g_{k-d})^{(d)}\circ F(f_d)$. On the other hand, since $\overline F(f)=(F(f_i))_{i\in\mathbb Z}$ and $\overline F(g)=(F(g_j))_{j\in\mathbb Z}$, the $k$-th component of $\overline F(g)\circ \overline F(f)$ is given by $\sum_{d\in\mathbb Z}F(g_{k-d})^{(d)}\circ F(f_d)$. Hence $\overline F(g\circ f)=\overline F(g)\circ \overline F(f)$.
\end{construction}

\begin{lem}\label{lem:push-down of equivalence}
    Let $\kappa$ and $\lambda$ be Kupisch functions. Assume  $F:(\RR,\kappa)\bpwf\to (\RR,\lambda)\bpwf$ preserves indecomposable objects and satisfies $F\circ (-)^{(1)}=(-)^{(1)}\circ F $.
    Let $\overline{F}$ be as in Construction~\ref{construction:push-down}. Then
    the following diagram commutes:
    \begin{displaymath}
    \xymatrix@C=8ex{
        (\RR,\kappa)\bpwf \ar[r]^-F \ar[d]_-{p_\bullet} & (\RR, \lambda)\bpwf \ar[d]^-{p_\bullet} \\
        (\bS^1,\kappa)\pwf \ar[r]_-{\overline{F}} & (\bS^1, \lambda)\pwf,
    }
    \end{displaymath}
where $p_\bullet$ denotes the push-down. Furthermore, if $F$ is an equivalence then so is $\overline{F}$.
\end{lem}
\begin{proof}
    
    First, by Construction~\ref{construction:push-down}, for interval modules $M_U\in  (\RR,\kappa)\bpwf$, we have $\overline F p_\bullet (M_U)=\overline F(\overline M_U)=\overline {F(M_U)}=p_\bullet F(M_U)$. Second, let $f:M_U\to M_V$ be a morphism of indecomposables in $(\RR,\kappa)\pwf$. Then $p_\bullet f=(f_i)_{i\in\mathbb Z}:\overline M_U\to \overline M_V$, where $f_0=f$ and $f_i=0$ for $i\neq 0$. Hence by Construction~\ref{construction:push-down}, $\overline F (p_\bullet f)=(F(f_i))$, where $F(f_0)=F(f)$ and $F(f_i)=0$ for $i\neq 0$, which coincides with $p_\bullet F(f)$. So
    the diagram commutes.
    
    Now assume $F$ is an equivalence.
    Then \[F: \bigoplus_{i\in\ZZ}\Hom_{(\RR,\kappa)\bpwf}(M_U, M_{V+i}) \to \bigoplus_{i\in\ZZ}\Hom_{(\RR,\lambda)\bpwf}(FM_U, FM_{V+i})\] is an isomorphism.
    So, \[\overline{F}:\Hom_{(\bS^1,\kappa)}(\overline{M}_U,\overline{M}_V)\to \Hom_{(\bS^1,\lambda)}(\overline{F}(\overline{M}_U),\overline{F}(\overline{M}_V))\] is also an isomorhism.
\end{proof}

\begin{thm}\label{iso1}
Let $\kappa$   be a Kupisch functions and $f\in\Homeo_+(\bS^1)$. Then $(\bS^1,\kappa)\pwf\cong(\bS^1,f_*\kappa)\pwf$.\\
\end{thm}
\begin{proof}
Since $\widetilde f:\RR\to\RR$ is an orientation preserving homeomorphism, it induces an equivalence $\widetilde F:(\RR,\kappa)\pwf\to(\RR,\tilde{f}_*\kappa)\pwf$, according to Theorem \ref{interval equiv}. It is clear $\widetilde F$ preserves indecomposables and $\widetilde F\circ (-)^{(1)}=(-)^{(1)}\circ \widetilde F$ because $\widetilde f$ has degree $1$.
Hence, by Lemma~\ref{lem:push-down of equivalence}, there is an induced equivalence $F:(\bS^1,\kappa)\pwf\to (\bS^1,f_*\kappa)\pwf$.
\end{proof}

\begin{exm}\label{s1 pushforward_exm}

(1) For any constant $c$, the shifted Kupisch function $\kappa(t-c)$ is a pushforward $f_*\kappa(t)$ via the circle map $f(t)=te^{2\pi ci}$. Therefore, $(\bS^1,\kappa(t-c))\pwf\cong(\bS^1,\kappa)\pwf$.

(2) If $\kappa$ and $\lambda$ are constant Kupisch functions with distinct values, then there is no orientation preserving homeomorphism $f$ on $\RR$ (or $\bS^1$) such that $f_*\kappa=\lambda$. It follows from Theorem~\ref{S1 equiv 2} that, in this case, $(\bS^1,\kappa)\pwf\not\cong(\bS^1,\lambda)\pwf$.

(3) If $\kappa=\frac{1}{8}\sin(2\pi t)+\frac{1}{2}$, $\lambda=\frac{1}{2}$, then $(\bS^1,\kappa)\pwf\not\cong(\bS^1,\lambda)\pwf$, since $(\bS^1,\kappa)\pwf$ contains non-period modules of infinite projective dimension whereas  $(\bS^1,\lambda)\pwf$ does not.   

 \end{exm}

\section{Separation Points and Conntectedness}  
In this section, we discuss the connectedness of $(\RR,\kappa)\pwf$ and of $(\bS^1,\kappa)\pwf$. We prove that the orthogonal components (Definition \ref{defn:connect}) of these categories are determined by the existence of certain discontinuities of the (pre-)Kupisch functions, which we call separation points (Definition~\ref{def:separation pt}). 

 \subsection{Separation Points}
Assume a real valued function $f(x)$ is discontinuous at $x_0$. Recall that $x_0$ is called a \textbf{discontinuity of first kind} if both $\lim\limits_{x\to x_0^-}f(x)$ and $\lim\limits_{x\to x_0^+}f(x)$ exist but are not equal.

We need the following well-known result for monotone functions from analysis, which is usually referred to as the Darboux--Froda Theorem \cite{R}.

  \begin{thm}
  A real valued monotone function defined on an interval has at most countably many discontinuities and each discontinuity must be of the first kind.
  \end{thm}

    Let $\kappa(t)$ be a pre-Kupisch function. Then $K(t)=\kappa(t)+t$ is an increasing function and $K(t)>t$. So $K(t)$ and hence $\kappa(t)$ has at most countably many discontinuities, all of the first kind. In particular, $\lim\limits_{t\to a^-}\kappa(t)$ and $\lim\limits_{t\to a^+}\kappa(t)$ exist for all real number $a$.
  
 We show that the category $(\RR,\kappa)\pwf$ decomposes into two orthogonal components when there is a sequence $\{K^n(t)\}_{n=1}^\infty$ such that $\displaystyle\lim_{n\to\infty} K^n(t)<\infty$. We characterize such limit points as separation points:
  
 \begin{defn}\label{def:separation pt}
Let $\kappa$ be a pre-Kupisch function on $I$. An interior point $c$ in $I\setminus\partial I$ is a {\bf separation point} if 
 \begin{enumerate}
 \item $\lim\limits_{t\to c^-}\kappa(t)=0$; \\ 
  \item $K(t)<c$, for $t\in I, t<c$.
 \end{enumerate}
 \end{defn}

Notice $\kappa$ may have discontinuities that are not separation points.
For example,
\begin{displaymath}
    \kappa(t) = \begin{cases}
    1 & t<0 \\
    2 & t\geq 0
    \end{cases}
\end{displaymath}
is a pre-Kupisch function with one discontinuity at 0 but 0 is not a separation point.

The following proposition follows from Definition~\ref{def:separation pt}.
\begin{prop}\label{k prop}
   Let $\kappa$ be a pre-Kupisch function on $I$.
  \begin{enumerate}

  \item If   $\lim\limits_{t\to c^-}\kappa(t)=0$, then $K(t)\leq c$ for $\forall t<c$. Furthermore, if the equality holds for some $t_0<c$, then $K(t)=c$ for $\forall t\in[t_0,c)$.
 \item If $c\in(t,K(t))$, for some  $t\in I$, then $\lim\limits_{t\to c^-}\kappa(t)\neq 0$.
  \item For any $t\in I$, there is no separation points in $(t,K(t)]$.
% \item $\cS(\kappa)$ is closed.
\end{enumerate}
\end{prop}

   Denote by $\cS(\kappa)$ the set of separation points. Since $\kappa(t)$ is discontinuous at any separation point, the cardinality of $\cS(\kappa)$ is at most countable.

  If $\kappa(t)$ is a pre-Kupisch function on $I$, then $K(t):I\to I$ is an increasing function. The value of $K(t)$ represents the socle of the indecomposable projective module at $t$. For any point $t$ and $n\geq 0$, there is an inequality $K^{n+1}(t)=\kappa(K^{n}(t))+K^{n}(t)>K^{n}(t)$. Hence  $\{K^n(t)\}_{n=0}^\infty$ is a strictly increasing sequence.

\begin{lem}\label{s prop}
Let $\kappa$ be a pre-Kupisch function on $I$. If $\lim\limits_{n\to\infty} K^n(a)=c\in I$, then $c$ is the minimum separation point which is greater than $a$.
 \end{lem}
  \begin{proof}
First we show that $c$ is a separation point. Since  $\lim\limits_{x\to c^-}\kappa(t)$ always exists, it can be computed by 
   $\lim\limits_{t\to c^-}\kappa(t)= \lim\limits_{n\to \infty}\kappa(K^n(a))=\lim\limits_{n\to \infty} K^{n+1}(a)-K^n(a)=0.$ On the other hand, for any $t<c$, there is some $n>0$ such that $t<K^n(a)<c$. Therefore, $K(t)\leq K^{n+1}(a)< \lim\limits_{n\to \infty}K^{n+1}(a)=c$ because $K^n(a)$ strictly increase. So $K(t)<c$ for $t<c$ and hence $c$ is a separation point.
   
Next,  according to Proposition \ref{k prop} (3), there is no separation point in intervals $(K^n(a),K^{n+1}(a)]$. So there is no separation point in the interval $(a,c)=\bigcup\limits_{n=0}^\infty(K^n(a),K^{n+1}(a)]$.
   \end{proof}

\begin{prop}\label{s cor}
 If $\cS(\kappa)=\emptyset$, then for all $t\in I$, $\lim\limits_{n\to \infty}K^n(t)=\begin{cases}\sup(I), &\sup(I)<+\infty\\ +\infty, & \text{otherwise}. \end{cases}$
  \end{prop}
  
  \subsection{Orthogonal decomposition of $(\RR,\kappa)$}
  Following \cite{K}, we have the following definition of connected additive category.
  \begin{defn}\label{defn:connect}
Let $\cA$ be an additive category and  $(\cA_i)_{i\in I}$ a family of full
additive subcategories. We have an orthogonal decomposition
$$\cA\cong\bigoplus\limits_{i\in I}\cA_i$$
of $\cA$ if $\cA=\sum_i \cA_i$ (each object can be written as a direct sum $\bigoplus\limits_{i\in I} X_i$ with $X_i\in \cA_i$) and $\Hom(X_i,X_j)=0$ for $X_i\in\cA_i$ and $X_j\in\cA_j$ and $i\neq j$. We call $\cA_i$'s the orthogonal components of $\cA$. The additive category $\cA$ is {\bf connected} if it admits no proper decomposition $\cA\cong\cA_1\oplus \cA_2$.
 \end{defn}

\begin{rmk}
A finite dimensional algebra $\Lambda$ is connected (i.e. $0$ and $1$ are the only central idempotents) if and only if the module category $\Lambda\modd$ is a connected additive category.
\end{rmk}

We remind the readers that there is a another notion of connected category commonly used in category theory: any two objects can be connected with a finite zigzags of (nonzero) morphisms.
Notice that every additive category is a ``connected category'' in this sense, since there are always nonzero morphisms $X\leftarrow X\oplus Y\rightarrow Y$ for any objects $X$ and $Y$.
To emphasize the difference, we always use the terminology ``connected additive category'' for our definition of connectedness as in Definition \ref{defn:connect}.
We compare these two notions in Lemma \ref{lem:hom orthogonal} below.

\begin{lem}\label{lem:hom orthogonal}
Let $\cA=\bigoplus\limits_{i\in I}\cA_i$ be an orthogonal decomposition. For any sequence of indecomposable objects $X_1, X_2, \cdots, X_n$ in $\cA$, if there are nonzero morphisms either $X_i\to X_{i+1}$ or $X_{i+1}\to X_i$ for all $0<i<n$, then $X_1$ and $X_n$ are in the same orthogonal component.
\end{lem}
\begin{proof}
	Since $X_1$ and $X_2$ are indecomposable, there must be $i,j$ such that $X_1\in\cA_i$ and $X_2\in\cA_j$.
	However, either $\Hom(X_1,X_2)\neq 0$ or $\Hom(X_2,X_1)\neq 0$ so we must have $i=j$. For the same reason, $X_2$ and $X_3$ and hence all $X_i$'s are in the same orthogonal component.  
\end{proof}

In this section, we are going to discuss the orthogonal decomposition of   $(\RR,\kappa)\pwf$ ( $(\mathbb S^1,\kappa)\pwf$) according to the (pre-)Kupisch function $\kappa$.

  \begin{lem}\label{I connect}
 Let $\kappa$ be a pre-Kupisch function on an interval $I$, then $(I,\kappa)\pwf$ is a connected additive category if and only if $\cS(\kappa)=\emptyset$.
  \end{lem}
  
   \begin{proof}
  First, assume $\cS(\kappa)\neq\emptyset$.
  If $c\in I$ is a separation point, let $\cC^-$ be the full subcategory of $(I,\kappa)\pwf$ consisting representations with indecomposable summands $M_U$ where $U\subseteq I\cap(-\infty,c)$, and   $\cC^+$ be the full subcategory of $(I,\kappa)\pwf$ consisting representations with indecomposable summands $M_U$ where $U\subseteq I\cap [c,+\infty)$.
  It is immediate that $\Hom(\cC^-,\cC^+)=0=\Hom(\cC^+,\cC^-)$. At the same time, for any interval module $M_U\in (I,\kappa)\pwf$, $U\subseteq [t,K(t)]\subseteq I$. According to Proposition \ref{k prop} (3), the interval $[t,t+K(t)]\subseteq I\cap(-\infty,c)$ or $I\cap[c,+\infty)$.
  So $M_U$ is in $\cC^-$ or in $\cC^+$.
  Hence $(I,\kappa)\pwf\cong\cC^-\oplus \cC^+$ is a disconnected additive category. 
  
Now, assume $\cS(\kappa)=\emptyset$. We first show that there is a finite sequence of nonzero morphisms between any two projective indecomposables $M_{[a,K(a)]}$ and $M_{[b,K(b)]}$.
Without loss of generality, assume $a<b$.
According to Proposition~\ref{s cor}, $\lim\limits_{n\to \infty}K^n(a)=\sup(I)$ or $+\infty$.
It follows that $b\in [K^n(a),K^{n+1}(a)]$ for some $n$. Hence there is a sequence of nonzero homomorphisms $$M_{[b,K(b)]}\to M_{[K^n(a),K^{n+1}(a)]}\to  M_{[K^{n-1}(a),K^{n}(a)]}\to\cdots\to M_{[a,K(a)]}.$$
  Second, notice that, for each $p<q\in I$, there are nonzero morphisms between indecomposables in $(I,\kappa)\pwf$:
\begin{displaymath}
    \xymatrix@C=2ex@R=2ex{
    & M_{[p,q]} \ar[dr] \\
    M_{(p,q]} \ar[ur] \ar[dr] & & M_{[p,q)}. \\
    & M_{(p,q)} \ar[ur]
    }
\end{displaymath}
and the nonzero projective cover: $M_{[p,K(p)]}\to M_{[p,q]}$. It follows that there is a zigzag sequence of nonzero morphisms between any two indecomposables in $(I,\kappa)\pwf$. According to Lemma \ref{lem:hom orthogonal}, any two indecomposables in $(I,\kappa)\pwf$ are in the same orthogonal component. Therefore, $(I,\kappa)\pwf$ is a connected additive category.
  \end{proof}

   Let $\kappa$ be a pre-Kupisch function on $\RR$. We now classify the orthogonal components of $(\RR,\kappa)\pwf$.
\begin{notation}\label{note:c star}
For a separation point $c\in\cS(\kappa)$, denote by $c^*=\lim\limits_{n\to \infty}K^n(c)$ when the limit exists.
\end{notation}
The following proposition is proven by straightforward computations.
\begin{prop}
Let $\kappa$ be a pre-Kupisch function and $c\in\cS(\kappa)\neq\emptyset$.
Then $\kappa|_{[c,c*)}$ is a pre-Kupisch function on $[c,c^*)$.
If $q=\max\cS(\kappa)$ exists, then $\kappa|_{[q,+\infty)}$ is a pre-Kupisch function on $[q,+\infty)$.
If $p=\min\cS(\kappa)$ exists, then $\kappa|_{(-\infty,p)}$ is a pre-Kupisch function on $(-\infty,p)$.
\end{prop}
   
\begin{lem}
Whenever they exist, the subcategories $\cC_c=([c,c^*),\kappa|_{[c,c*)})\pwf$, $\cC_{max}=([q,+\infty),\kappa|_{[q,+\infty)})\pwf$, and $\cC_{min}=((-\infty,p),\kappa|_{(-\infty,p)})\pwf$ are connected additive category.
\end{lem}
\begin{proof}
It follows from the definition that $\cS(\kappa|_{[c,c*)})=\cS(\kappa)\cap(c,c^*)$. 
However, according to Lemma \ref{s prop}, $\cS(\kappa)\cap(c,c^*)=\emptyset$. Hence, by Lemma \ref{I connect}, the subcategory $\cC_c=([c,c^*),\kappa|_{[c,c*)})\pwf$ is a connected additive category.
  
   Similarly, if $q=\max\cS(\kappa)$ (respectively $p=\min\cS(\kappa)$) exists we have  $\cC_{max}=([q,+\infty),\kappa|_{[q,+\infty)})\pwf$ (respectively $\cC_{min}=((-\infty,p),\kappa|_{(-\infty,p)})\pwf$)  is a connected additive category.
   \end{proof}
   
    Therefore, if $\cS(\kappa)\neq\emptyset$, whichever of $\cC_c$, $\cC_{max}$, and $\cC_{min}$ exist are all of the orthogonal components of $(\RR,\kappa)\pwf$.
   Up to equivalence (Corollary \ref{R equiv 1}), we have the following orthogonal decomposition of $(\RR,\kappa)\pwf$.
   
  \begin{thm}\label{R connect}
Let $\kappa$ be a pre-Kupisch function on $\RR$ and $\cS(\kappa)$ the set of separation points. Then
$$(\RR,\kappa)\pwf\cong \begin{cases}
\bigoplus\limits_{c\in\cS(\kappa)}(\bR^{\geq 0},\kappa_c)\pwf & \inf(\cS(\kappa))=-\infty \\
(\RR,\kappa_0)\pwf\oplus\bigoplus\limits_{c\in\cS(\kappa)}(\bR^{\geq 0},\kappa_c)\pwf &\inf(\cS(\kappa))>-\infty,
\end{cases}$$ 
where $\kappa_0$ and $\kappa_c$'s are pre-Kupisch functions on $\RR$ and $\bR^{\geq 0}$, respectively, which satisfy $\cS(\kappa_0)=\cS(\kappa_c)=\emptyset$.
    \end{thm}

 \subsection{Orthogonal decomposition of $(\bS^1,\kappa)$}
 
Let $\kappa$ be a Kupisch function and recall $K(t)=\kappa(t)+t$. All the properties still hold for $\kappa$ as a pre-Kupisch function on $\RR$. Furthermore, since $\kappa$ is 1-periodic, it has the following property.

 \begin{lem}
 If $\kappa$ is a Kupisch function, then $c\in\cS(\kappa)$ if and only if $c+1\in\cS(\kappa)$. 
 \end{lem}
 
 Therefore $\cS(\kappa)= \{c+n\mid c\in\cS(\kappa)\cap[0,1),\ n\in\mathbb Z$\}.
 
If $\cS(\kappa)=\emptyset$, by Theorem \ref{I connect}, $(\RR,\kappa)\pwf$ is connected. Therefore, under the push-down, $(\bS^1,\kappa)\pwf$ is connected.

 \begin{lem}
Let $\kappa$ be a Kupisch function, $c\in\cS(\kappa)$, and $I=[c,c+1)$.
Then $\kappa|_I$ is a pre-Kupisch function on $I$ and $(\bS^1,\kappa)\pwf\cong(I,\kappa|_I)\pwf$.
   \end{lem}
   
   \begin{proof}
   Let $F:(I,\kappa|_I)\pwf\to(\bS^1,\kappa)\pwf$ be the push-down functor sending $F(M_U)=\overline{M}_U$. For intervals $U$, $V\subseteq [c,c+1)$, $U\cap (V+i)=\emptyset$ for $i\neq 0$. Hence $\Hom_{\bS^1}(\overline{M}_U,\overline{M}_V)\cong  \oplus_{i\in\mathbb Z}\Hom_{\RR}(M_U,M_{V+i})=\Hom_\RR(M_U,M_V)$. Therefore, $F$ induces an equivalence $(I,\kappa|_I)\pwf\cong(\bS^1,\kappa)\pwf$.
   \end{proof}
Recall that for any $c\in I\subseteq\RR$, $c^*=\lim\limits_{n\to \infty}K^n(c)$ when the limit exists (Notation~\ref{note:c star}).
 \begin{thm}\label{S1 components}
Suppose $c_0\in\cS(\kappa)\neq \emptyset$. Then $$(\bS^1,\kappa)\pwf\cong \bigoplus\limits_{c\in\cS(\kappa)\cap[c_0,c_0+1)}([c,c^*),\kappa|_{[c,c^*)})\pwf\cong\bigoplus\limits_{c\in\cS(\kappa)\cap[c_0,c_0+1)}(\bR^{\geq 0},\kappa_c)\pwf$$
for some pre-Kupisch functions $\kappa_c$, satisfying $\cS(\kappa_c)=\emptyset$.
 \end{thm}
 
 Notice that $\cS(\kappa)\cap[c_0,c_0+1)$ is in bijection with $\cS(\kappa)\cap[0,1)$, since $\kappa$ is $1$-periodic.
 
\begin{cor}[to Theorem~\ref{S1 components}]
 $(\bS^1,\kappa)\pwf$ is a connected additive category if and only if $|\cS(\kappa)\cap[0,1)|\leq 1$. 
\end{cor}

   \begin{exm}
  If $\kappa(t_0)\geq 1$ for some $t_0$, then $(\mathbb S^1,\kappa)\pwf$ is a connected additive category
  \end{exm}
 
 In Section~\ref{sec:homeomorphisms and equivalences} we show an example of a category $(\bS^1,\kappa)\pwf$ which is not a connected additive category (Example~\ref{n piece}).

\section{Orientation Preserving Homeomorphisms and Equivalences}\label{sec:homeomorphisms and equivalences}

In this section, we prove the converse of Theorems~\ref{interval equiv}~and~\ref{iso1} for (pre-)Kupisch functions without separation points. First, we summarize some  properties of equivalence functors, which we use later.

\begin{lem}\label{lem:equivalence}
Let $F:\mathcal C\to \mathcal D$ be an equivalence between abelian categories. Then
\begin{enumerate}
\item $F$ is exact;
\item $F$ preserves isomorphism classes;
\item $F$ preserves simple objects;
\item $F$ preserves socle and top; (i.e. If $S$ is a simple sub-object/quotient object of $X$, then $F(S)$ is a simple sub-object/quotient object of $F(X)$.)
\item $F$ preserves indecomposability.
\item $F$ preserves bricks.
\item $F$ preserves projective objects and preserves injective objects.
\item $F$ preserves subfactors; (i.e, if $Y$ is a subobject of a quotient object $X$ in $\mathcal{C}$ then $F(Y)$ is a subobject of a quotient object of $F(X)$ in $\mathcal{D}$.)
\end{enumerate}
\end{lem}

Next, we show that if an equivalence functor sends an indecomposable interval module $M_U$ to an indecomposable interval module $N_V$, then these two intervals $U$ and $V$ have the same ``open-close'' type.
\begin{lem}
Let $F:(I,\kappa)\pwf\to(J,\lambda)\pwf$ be an equivalence of categories where $\kappa$ and $\lambda$ are pre-Kupisch functions on intervals $I$ and $J$, respectively.
Then for any $[x,y]\subseteq [x,x+\kappa(x)]\subseteq I$, where $x<y$, there is some $[a,b]\subseteq [a,a+\lambda(a)]\subseteq J$, such that $F(M_{[x,y]})\cong N_{[a,b]}$, $F (M_{(x,y]})\cong N_{(a,b]}$, $F( M_{[x,y)})\cong N_{[a,b)}$, and $F (M_{(x,y)})\cong N_{(a,b)}$.
\end{lem}
\begin{proof}
Since equivalence functors preserve indecomposability, $F(M_{[x,y]})$ is an indecomposable representation in $(J,\lambda)\pwf$. The equivalence functors also preserve the top and socle of a representation. Hence $F(M_{[x,y]})\cong N_{[a,b]}$ for some closed interval $[a,b]$. 

Because equivalent functors are exact, applying $F$ to the exact sequence $0\to M_{(x,y]}\to M_{[x,y]}\to M_{[x,x]}\to 0$, we have $F(M_{(x,y]})\cong \ker (N_{[a,b]}\to N_{[a,a]})\cong N_{(a,b]}$. The other statements follow in a similar manner.
\end{proof}

A similar conclusion also holds for indecomposable string modules.
\begin{lem}\label{lem:circle equivalence on strings}
Let $F:(\bS^1,\kappa)\pwf\to(\bS^1,\lambda)\pwf$ be an equivalence of categories where $\kappa$ and $\lambda$ are Kupisch functions.
Then for any $[x,y]\subseteq [x,x+\kappa(x)]$, where $x<y$, there is some $[a,b]\subseteq [a,a+\lambda(a)]$, such that $F(\overline M_{[x,y]})\cong \overline N_{[a,b]}$, $F (\overline M_{(x,y]})\cong \overline N_{(a,b]}$, $F( \overline M_{[x,y)})\cong \overline N_{[a,b)}$, and $F (\overline M_{(x,y)})\cong \overline N_{(a,b)}$.
\end{lem}

\subsection{Proofs of the Theorems}
Now we proceed to prove the converse to Theorems~\ref{interval equiv}~and~\ref{iso1} in the case that there are no separation points.
Recall that for a pre-Kupisch function $\kappa$, $K(t) = \kappa(t)+t$.

\begin{thm}\label{R equiv 2}
Let $F:(I,\kappa)\pwf\to(J,\lambda)\pwf$ be an equivalence of categories where $\kappa$ and $\lambda$ are pre-Kupisch functions on intervals $I$ and $J$, respectively, such that $\cS(\kappa)=\cS(\lambda)=\emptyset$.
Then $F$ induces an orientation preserving homeomorphism $f\in\Homeo_+(I,J)$ such that $\lambda=f_* \kappa$.
\end{thm}

\begin{proof}
We denote by $M_{U}$ and $N_{V}$ interval modules of  $(I,\kappa)\pwf$ and $(J,\lambda)\pwf$, respectively.   

Since the equivalence functor $F$ preserves simple representations, for any simple representation $M_{[x,x]}$, $F(M_{[x,x]})\cong N_{[a,a]}$ for some $a\in J$. Define $f(x)=a$ and we are going to show $f:I\to J$ is an orientation preserving homeomorphism. 

The fact $f$ is a bijection follows immediately from Lemma~\ref{lem:equivalence}(3) and the fact that $F$ is dense. 

If $x<y$ in $I$, then there exists a natural number $n$ such that $K^{n-1}(x)<y\leq K^n(x)$.
So, there is a sequence of nonzero morphisms in $(I,\kappa)\pwf$:
 $$M_{[y,y]}\to M_{[K^{n-1}(x),y]}\to  M_{[K^{n-2}(s),K^{n-1}(x)]}\to\cdots\to M_{[x,K(x)]}\to M_{[x,x]} .$$
Since the equivalence functor $F$ also preserves indecomposable representations, for any interval module $M_U$, $F(M_U)=N_{V}$ for some interval $V\subseteq J$. Applying $F$ to the above sequence morphisms, we have a sequence of nonzero morphisms in $(J,\lambda)\pwf$:
$$N_{[f(y),f(y)]}\to N_{V_{n-1}}\to N_{V_{n-2}}\to \cdots\to N_{V_0}\to N_{[f(x),f(x)]}.$$

Hence, $[f(x),f(x)]\cap_L V_0\neq \emptyset$, $V_0\cap_LV_1\neq \emptyset, \cdots, V_{n-1}\cap_L [f(y),f(y)]\neq \emptyset$, which implies $f(x)\leq f(y)$. So, $f$ is increasing and therefore $f\in\Homeo_+(I,J)$.

Now we show $\lambda= f_*\kappa$.
Notice $F$ preserves the top and socle of any interval module: $F(M_{[a,b]})=N_{[f(a),f(b)]}$.
Moreover, since $F$ also preserves projective representations. For each indecomposable projective $M_{[t,\kappa(t)+t]}$, $F(M_{[t,t+\kappa(t)]})=N_{[f(t), f(t)+\lambda(f(t))]}$. So we have $f(t+\kappa(t))=f(t)+\lambda(f(t))$, which means $\lambda= f_*\kappa$.

\end{proof}

\begin{thm}\label{S1 equiv 2}
Let $F:(\bS^1,\kappa)\pwf\to(\bS^1,\lambda)\pwf$ be an equivalence of categories where $\kappa$ and $\lambda$ are Kupisch functions such that $\cS(\kappa)=\cS(\lambda)=\emptyset$.
Then $F$ induces an orientation preserving  homeomorphism $f\in\Homeo_+(\bS^1)$ such that $\lambda=f_* \kappa$.
\end{thm}

Before we prove the theorem we provide a useful construction and some technical results.

We denote by $\overline M_{U}$ and $\overline N_{V}$ interval modules of  $(\bS^1,\kappa)\pwf$ and $(\bS^1,\lambda)\pwf$ respectively.  
\begin{construction}\label{circle f}
  Since the equivalence functor $F$ preserves simple objects, there is a unique $h\in[0,1)$ such that $F(\overline{M}_{[0,0]})=\overline{N}_{[h,h]}$.
  For all $n\in\ZZ$, define $\tilde f(n) = h+n$.
  
  For each $t\in (0,1)$ there is a unique $b\in(h,h+1)$ such that
  $F(\overline M_{[t,t]})\cong \overline N_{[b,b]}$.
  For all $t\in(0,1)$ and $n\in\ZZ$, define $\tilde f(t+n) = b+n$.
  Define $f:\bS^1\to\bS^1$ as $f(e^{2\pi i t})=e^{2\pi i b}$, for each $t\in[0,1)$.
  
  Notice that $\tilde f$ is a lift of $f$.
\end{construction}
  Since $F$ is an equivalence, $f$ is a bijection. 
  In the following, we show that $f\in \Homeo_+(\bS^1)$ by showing $\tilde f:\RR\to \RR$ is an increasing bijection (we cannot apply the homotopy lifting theorem directly, since \emph{a priori} we do not know $f$ is continuous).

\begin{lem}\label{g increasing}
Let $\tilde f:\RR\to \RR$ be the function defined in Construction~\ref{circle f}. Then 
\begin{enumerate}
\item $\tilde f(t+1)=\tilde f(t)+1$.
\item $\tilde f$ is a bijection.
\item $\tilde f$ is increasing. 
\end{enumerate}
\end{lem}

\begin{proof}

(1) follows immediately from the definition.

(2) Since $f$ is a bijection, restricting on each interval $[n,n+1)$, the function $\tilde f|:[n,n+1)\to [\tilde{f}(0)+n,\tilde{f}(0)+n+1)$ is a bijection. Hence $\tilde f:\RR\to \RR$ is a bijection.

(3) First, we prove $\tilde f$ is increasing locally. 

{\bf Claim:} For any $0\leq s<t<1$ and $t\leq K(s)= s+\kappa(s)$, it follows that $\tilde f(s)\leq \tilde f(t)$.

Indeed, consider the module $\overline M_{[s,t]}$ in $(\bS^1,\kappa)\pwf$. According to Lemma \ref{lem:circle equivalence on strings}, $F\overline M_{[s,t]}\cong\overline N_{[a,b]}$, for some interval $[a,b]$. Since $F$ preserves top and socle, $F\overline M_{[s,s]}\cong\overline N_{[a,a]}$ and $F\overline M_{[t,t]}\cong\overline N_{[b,b]}$. Hence $\tilde f(s)=a-i$ and $\tilde f(t)=b-j$ for some $i,j\in\mathbb Z$.

Next, because $0\leq t-s<1$, $\overline M_{[s,t]}$ is a brick.
Hence, by Lemma~\ref{lem:equivalence}(6), $\overline N_{[a,b]}$ is also a brick.
This implies $0\leq b-a<1$ and so $i-j\leq \tilde f(t)-\tilde f(s)<i-j+1$. But, $|\tilde f(t)-\tilde f(s)|<1$, since both $\tilde f(s)$ and $\tilde f(t)\in [\tilde{f}(0),\tilde{f}(0)+1)$.
Thus, $i-j=-1$ or $0$.

If $i-j=0$, it follows that  $\tilde f(s)\leq \tilde f(t)$.

We show $i-j$ cannot be $-1$ by finding a contradiction.
If $i-j=-1$, it follows that $\tilde f(t)< \tilde f(s)\leq\tilde f(t)+1$. Also from the definition of $\tilde f$, we have an inequlity:
$$
\tilde{f}(0)\leq \tilde f(t)<\tilde f(s)<\tilde{f}(0)+1\leq \tilde f(t)+1.
$$
Notice that, in this case, $\overline N_{[a,b]}\cong\overline N_{[\tilde f(s),\tilde f(t)+1]}$ and, from the above inequality, $\overline N_{[\tilde{f}(0)+1,\tilde{f}(0)+1]}\cong \overline N_{[\tilde{f}(0),\tilde{f}(0)]}$ is a subfactor
of $\overline N_{[a,b]}$.
Therefore $\overline M_{[0,0]}$ must be a subfactor of $\overline M_{[s,t]}$, which forces $s=0$. But then $\tilde f(s)=\tilde f(0)$ contradicts the above inequality.

Second, we prove that $\tilde f$ is increasing on $[0,1)$.
Since $\cS(\kappa)=\emptyset$, for $0\leq s<t<1$, there is a natural number $m$ such that $K^{m}(s)<t\leq K^{m+1}(s)$. Apply the claim for pairs $K^{i-1}(s)$ and $ K^i(s)$ for $i=1,\cdots, m$ and for the pair $K^m(s)$ and $t$, we have $\tilde f(s)\leq \tilde f(K(s))\leq \cdots\leq \tilde f(K^{m}(s))\leq \tilde f(t)$. Therefore $\tilde f$ is increasing on $[0,1)$.

Last, combining (1) with the fact that $\tilde f$ is increasing on $[0,1)$, it follows that $\tilde f$ is increasing on $\RR$.
\end{proof}

\begin{cor}[to Lemma~\ref{g increasing}]
 The map $f$ defined in Construction \ref{circle f} is an orientation preserving homeomorphism in $\Homeo_+(\bS^1)$ with lift $\tilde f\in\Homeo_+(\RR)$.
\end{cor}

\begin{lem}\label{F interval}
If $\overline M_{[s,t]}$ is a string module in $(\bS^1,\kappa)\pwf$, then $F(\overline M_{[s,t]})\cong\overline N_{[\tilde f(s),\tilde f(t)]}$.
\end{lem}

\begin{proof}
Suppose $F(\overline M_{[s,t]})=\overline N_{[a,b]}$. Then $a=\tilde f(s)+i$ and $b=\tilde f(t)+j$. It suffice to show $i=j$.

Suppose $d\leq t-s<d+1$ for some integer $d$. This is equivalent to say $\dim_{\bk}\End_{\bS^1}(\overline M_{[s,t]})=d+1$. Hence $\dim_{\bk}\End_{\bS^1}(\overline N_{[a,b]})=d+1$ and therefore $b-a\in[d,d+1)$. That is
$$
 \tilde f(t)-\tilde f(s)\in[d+i-j, d+1+i-j).
$$

Now, since $s+d\leq t<s+d+1$ and $\tilde f$ is increasing, we have $$\tilde f(s)+d=\tilde f(s+d)\leq \tilde f(t)<\tilde f(s+d+1)=\tilde f(s)+d+1.$$ That is
$
 \tilde f(t)-\tilde f(s)\in[d,d+1).
$
This forces $i-j=0$, which finishes the proof.
\end{proof}

\noindent {\bf Proof of Theorem \ref{S1 equiv 2}}

It remains to show $\lambda=f_*\kappa$. Notice the equivalence $F$ preserves indecomposable projective representations. According to Lemma \ref{F interval} $F(\overline M_{[x,x+\kappa(x)]})\cong\overline N_{[\tilde f(x), \tilde f(x+\kappa(x))]}\cong \overline N_{[\tilde f(x), \tilde f(x)+\lambda(\tilde f(x))]}$. So, $ \tilde f(x+\kappa(x))=\tilde f(x)+\lambda(\tilde f(x))$. That is, $$\lambda(t)=\tilde f(\tilde f^{-1}(t)+\kappa\circ\tilde f^{-1}(t))-t.$$
By definition, $\lambda=f_*\kappa$. \hfill$\Box$

\begin{cor}[to Theorem~\ref{S1 equiv 2}]\label{cor:kupisch eq lift}
If  $F:(\bS^1,\kappa)\pwf\to(\bS^1,\lambda)\pwf$ is an equivalence, then $F$ induces an equivalence functor $\widetilde{F}:(\RR,\kappa)\pwf\to(\RR,\lambda)\pwf$ such that $F(\overline M_U)=\overline{\widetilde{F}(M_U)}$ for each interval module $\overline M_U\in(\bS^1,\kappa)\pwf$.
\end{cor}

\begin{proof}
This follows since $\widetilde{f}$ in Construction~\ref{circle f} is in $\Homeo_+(\RR)$, by Lemmas~\ref{lem:push-down of equivalence}~and~\ref{g increasing}.
\end{proof}

 We make some remarks about Theorem \ref{S1 equiv 2} in the cases when the Kupisch functions have at least one separation points.  First, let us fix some notations: for any map $T:[0,1)\to[0,1)$,  we can extend it to $\widetilde T:\RR\to\RR$ by defining ${\widetilde T}(t)=T(t)+n$ for $t\in [n,n+1)$, $n\in\mathbb Z$ and push it down to $\bar T:\bS^1\cong\RR/\mathbb Z\to \RR/\mathbb Z$ by defining $\bar T([t])=[\widetilde T(t)]$, where $[x]$ stands for the coset $x+\mathbb Z \in \RR/\mathbb Z$.
\begin{rmk}\label{rmk:interval exchange}
(1) When $\cS(\kappa)\cap[0,1)$ is a finite set,  Theorem \ref{S1 equiv 2} can be modified using interval exchanging maps. That is, the equivalence $F:(\bS^1,\kappa)\pwf\to(\bS^1,\lambda)\pwf$ induces a homeomorphism $f:\bS^1\to\bS^1$ such that for some interval exchanging map $T:[0,1)\to[0,1)$, $\bar T\circ f$ is orientation preserving and $\lambda=[(\bar T\circ f)_* \kappa]\circ \widetilde T$. 
(See an example of interval exchanging map in Figure~\ref{fig:interval exchange} and see \cite{V} for the formal definition.)
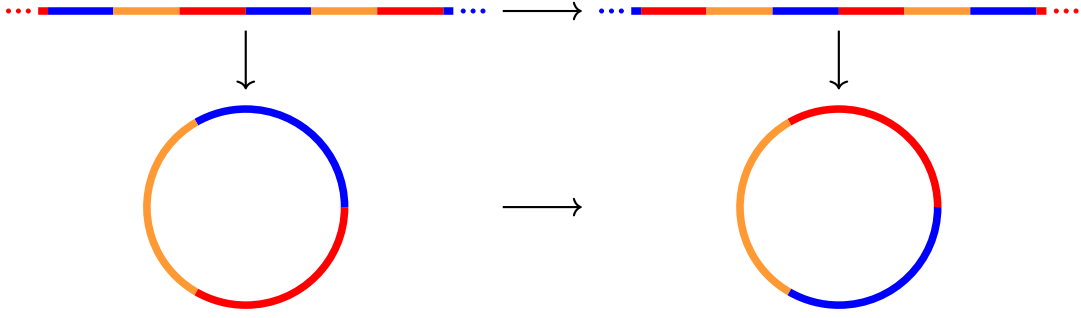
\begin{figure}[h]
    \centering
    \begin{tikzpicture}[scale=1.3]
        %R kappa
        \foreach \x in {0,2}
        {
            \draw[line width = 1mm, blue] (\x ,3) -- (\x+.66 ,3);
            \draw[line width = 1mm, orange!80!white] (\x+.66 ,3) -- ( \x+1.33 ,3);
            \draw[line width = 1mm, red] (\x+1.33 ,3) -- ( \x+2 ,3);
        }
        \draw[line width = 1mm, red] (0 ,3) -- (-0.1 ,3);
        \draw[line width = 1mm, blue] (4 ,3) -- (4.1 ,3);
        \foreach \x in {0.1, 0.2, 0.3}
        {
            \filldraw[fill=red, draw=red] (-0.1-\x ,3) circle[radius=.2mm];
            \filldraw[fill=blue, draw=blue] (4.1+\x ,3) circle[radius=.2mm];
        }
        %R lambda
        \foreach \x in {6,8}
        {
            \draw[line width = 1mm, red] (\x ,3) -- (\x+.66 ,3);
            \draw[line width = 1mm, orange!80!white] (\x+.66 ,3) -- ( \x+1.33 ,3);
            \draw[line width = 1mm, blue] (\x+1.33 ,3) -- ( \x+2 ,3);
        }
        \draw[line width = 1mm, blue] (6 ,3) -- (5.9 ,3);
        \draw[line width = 1mm, red] (10 ,3) -- (10.1 ,3);
        \foreach \x in {0.1, 0.2, 0.3}
        {
            \filldraw[fill=blue, draw=blue] (5.9-\x ,3) circle[radius=.2mm];
            \filldraw[fill=red, draw=red] (10.1+\x ,3) circle[radius=.2mm];
        }
        %S kappa
        \draw[line width = 1mm, blue] (3,1) arc (0:120:1);
        \draw[line width = 1mm, orange!80!white] (1.5,1.866) arc (120:240:1);
        \draw[line width = 1mm, red] (3,1) arc (360:240:1);
        %S lambda
        \draw[line width = 1mm, red] (9,1) arc (0:120:1);
        \draw[line width = 1mm, orange!80!white] (7.5,1.866) arc (120:240:1);
        \draw[line width = 1mm, blue] (9,1) arc (360:240:1);
        \foreach \y in {3,1}
            \draw[thick, ->] (4.6,\y) -- (5.4,\y);
        \foreach \x in {2,8}
            \draw[thick, ->] (\x,2.8) -- (\x,2.2);
.,.m     \end{tikzpicture}
    \caption{An example of an interval exchanging map for intervals that are periodic in $\RR$ as in Remark~\ref{rmk:interval exchange}.
    We push the interval exchange down to the circle using the periodicity.
    The orange interval is fixed.
    The red and blue intervals are swapped.}\label{fig:interval exchange}
\end{figure}

(2) In particular, since when $|\cS(\kappa)\cap[0,1)|\leq 2$, $\bar T$ is just a circle rotation, and $\widetilde T=x+\theta$, for some rotation number $\theta=T(0)$. Together with Example \ref{s1 pushforward_exm} (1), this implies $\lambda(t)=[(\bar T\circ f)_* \kappa (t+\theta)$ is still a push-forward of $\kappa$ by an orientation-preserving homeomorphism. 
\end{rmk}

\begin{exm}\label{n piece} Let $n\in\NN\setminus\{0\}$ and let $f_n:  \RR^{\geq 0}\to [0,\frac{1}{n})$, given by $t\mapsto \frac{t}{n(1+t)}$ be an orientation preserving homeomorphism.
Let $\lambda(t)=t+1$ be a pre-Kupisch function and define a Kupisch function $\kappa_n(t)=(f_*\lambda)(t-\frac{k}{n})$ for $\frac{k}{n}\leq t<\frac{k+1}{n}, k\in\ZZ$.
Explicitly, 
 $\kappa_n(t)= \frac{k+1}{2n}-\frac{t}{2}$ for $\frac{k}{n}\leq t<\frac{k+1}{n}, k\in\ZZ$.
We see that $\cS(\kappa_n)=\frac{1}{n}\mathbb Z$.
The category $(\mathbb S^1,\kappa_n)\pwf$ is disconnected for $n>1$.
See Figure~\ref{fig:disconnected S} for a picture of the orthogonal components in $(\bS^1,\kappa_n)\pwf$, for $n\in\{1,2,3\}$.
\begin{figure}[h]
    \centering
    \begin{tikzpicture}[scale=1.3]
        %k_1
        \draw[thick] (1,0) arc (0:180:1) arc (180:360:0.95);
        \filldraw[fill=black] (1,0) circle[radius=.4mm];
        \filldraw[fill=white] (.9,0) circle[radius=.4mm];
        \draw (0,-1) node [anchor=north] {$(\bS^1,\kappa_1)\pwf$};
        %k_2
        \draw[thick] (4.05,0) arc (0:180:1);
        \draw[thick] (1.95,0) arc (180:360:1);
        \foreach \x in {4.05, 1.95}
            \filldraw[fill=black] (\x,0) circle[radius=.4mm];
        \foreach \x in {2.05,3.95}
            \filldraw[fill=white] (\x,0) circle[radius=.4mm];
        \draw (3,-1) node [anchor=north] {$(\bS^1,\kappa_2)\pwf$};
        %k_3
        \draw[thick] (7.1,0) arc (0:122:1.035);
        \draw[thick] (5.45,0.952) arc (120:242:1.035);
        \draw[thick] (5.45,-0.952) arc (240:362:1.035);
        \filldraw[fill=black] (7.1,0) circle[radius=.4mm];
        \filldraw[fill=white] (7,0) circle[radius=.4mm];
        \filldraw[fill=black] (5.45,0.952) circle[radius=.4mm];
        \filldraw[fill=white] (5.5,0.866) circle[radius=.4mm];
        \filldraw[fill=black] (5.45,-0.952) circle[radius=.4mm];
        \filldraw[fill=white] (5.5,-0.866) circle[radius=.4mm];
        \draw (6,-1) node [anchor=north] {$(\bS^1,\kappa_3)\pwf$};
    \end{tikzpicture}
    \caption{Orthogonal components for Kupisch functions $\kappa_n$ in Example~\ref{n piece}, where $n\in\{1,2,3\}$.}
    \label{fig:disconnected S}
\end{figure}
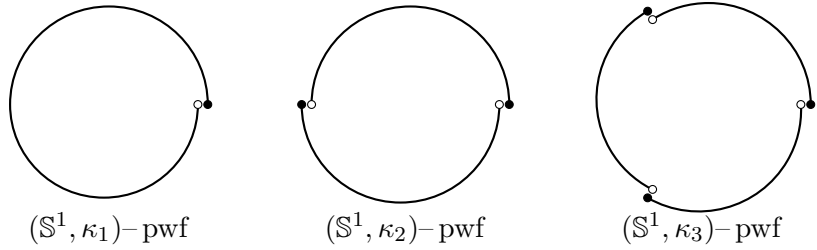
By Theorem \ref{S1 components}, $(\bS^1,\kappa_n)\pwf$ is equivalent to the direct sum of $n$ copies of $(\RR^{\geq 0},\lambda)\pwf.$

There is an auto-equivalence $F:(\bS^1,\kappa_3)\pwf\to (\bS^1,\kappa_3)\pwf$ defined the following way on objects:
$$
F\overline M_U=
\begin{cases}
\overline M_{U-\frac{1}{3}} & U\subseteq [0,\frac{1}{3})\\
\overline M_{U} & U\subseteq [\frac{1}{3},\frac{2}{3})\\
\overline M_{U+\frac{1}{3}} & U\subseteq [\frac{2}{3},1).
\end{cases}
$$
\end{exm}

Consider the interval exchanging map 
$$
T(t)=\begin{cases}
t-\frac{1}{3} & t\in [0,\frac{1}{3})\\
t & t\in [\frac{1}{3},\frac{2}{3})\\
t+\frac{1}{3} & t\in [\frac{2}{3},1).
\end{cases}
$$ 
 Restricting the functor $F$ on the simple objects induces a circle map $f=\bar T$, which is not an orientation preserving homeomorphism. It is easy to verify that $\bar T^2=\id$ and  $\kappa(t)=\kappa(\widetilde T(t))$.

\vskip 5pt

According to Theorem \ref{R connect} and Theorem \ref{S1 components}, the orthogonal components of categories $(\RR,\kappa)\pwf$ or $(\bS^1,\lambda)\pwf$ are of the following forms:
\begin{enumerate}[label=(\roman*)]
\item $(\RR^{\geq 0},\mu)\pwf$ where $\cS(\mu)=\emptyset$;
\item $(\RR,\nu)\pwf$ where $\cS(\nu)=\emptyset$; and
\item $(\bS^1, \lambda)\pwf$ where $\cS(\lambda)=\emptyset$.
\end{enumerate}

We verify that categories of the forms (i), (ii), (iii) above are pairwise non-equivalent:
 
\begin{prop}\label{prop:types of categories}
Let  $\mu$, $\nu$ be pre-Kupisch functions on $\RR^{\geq 0}$ and $\RR$, respectively, and $\lambda$ be a Kupisch function. Assume $\cS(\lambda)=\cS(\mu)=\cS(\nu)=\emptyset$. Then categories (i) $(\RR^{\geq 0},\mu)\pwf$, (ii) $(\RR,,\nu)\pwf$ and (iii) $(\bS^1, \lambda)\pwf$  are mutually non-equivalent. 
\end{prop}
\begin{proof}

Since $\RR^{\geq 0}$ and $\RR$ are non-homeomoprhic, categories of the forms (i) and (ii) are not equivalent, by Theorem \ref{R equiv 2}.

Next, denote by $L(t)=\lambda(t)+t$. By Proposition~\ref{s cor}, there is $N>0$ such that $1\in(L^N(0),L^{N+1}(0)]$. Then there is a sequence of nonzero non-isomorphisms between indecomposable $\bS^1$-representations compatible with $\lambda$. 
$$
\overline{M}_{[0,0]}\cong\overline{M}_{[1,1]}\to\overline{M}_{[L^N(0),1]}\to\overline{M}_{[L^{N-1}(0),L^{N}(0)]}\to\cdots\to\overline{M}_{[L(0),L^2(0)]}\to\overline{M}_{[0,L(0)]}\to\overline{M}_{[0,0]},
$$
where $\overline{M}_{[0,0]}$ is simple. But there is no such a sequence for any simple $\RR$-representation. Therefore $(\bS^1, \lambda)\pwf$ is not equivalent to any subcategory of $\RR\pwf$, hence not equivalent to categories (i) and (ii)
\end{proof}

Proposition~\ref{prop:types of categories} may not be true if the (pre-)Kupisch functions have separation points, as shown in the following example.

\begin{exm}\label{N=Z}
For all $n\in\ZZ$ and $t\in[n,n+1)$, let $\nu(t)=\frac{n+1-t}{2}$ and let $\mu=\nu|_{[0,\infty)}$. Notice $\nu$ and $\mu$ are pre-Kupisch functions on $\RR$ and $\RR^{\geq 0}$, respectively.
We see that $\cS(\nu)=\mathbb Z$ and $\cS(\mu)=\mathbb N$. The connected orthogonal components of both $(\RR,\nu)\pwf$ and $(\RR^{\geq 0},\mu)\pwf$ are equivalent to $([0,1),\nu|_{[0,1)})\pwf\cong (R^{\geq 0}, \eta)\pwf$, for some pre-Kupisch function $\eta$.

To compute an $\eta$ explicitly, choose a homeomorphism $f:[0,1)\to [0,\infty)$, say $f(t)=\frac{t}{1-t}$, and take the push-forward  $\eta(t)=f_*(\nu|_{[0,1)})=2t+1$.
 Then
$$
(\RR,\nu)\pwf\cong  ((\RR^{\geq 0}, \eta)\pwf)^\mathbb Z\cong  ((\RR^{\geq 0}, \eta)\pwf)^\mathbb N\cong (\RR^{\geq 0},\mu)\pwf.
$$ 
\end{exm}

\subsection{Relation to Dynamical Systems}
One concludes that the classification of categories $(\bS^1,\kappa)\pwf$ and $(\RR,\kappa)\pwf$ up to categorical equivalence reduces to the classification of categories of types  $(\RR,\kappa)\pwf$, $(\RR^{\geq 0},\kappa)\pwf$, or $(\bS^1,\kappa)\pwf$, where $\cS(\kappa)=\emptyset$. We want to point out a further survey about the classification problem is related with topological dynamical systems. 
 
A topological dynamical system $(X,\sigma)$ contains a topological space $X$ and a self-map $\sigma:X\to X$. A topological conjugacy between systems $(X,\sigma)$ and $(Y,\tau)$ is a homeomorphism $f:X\to Y$ such that $f\circ \sigma=\tau\circ f$. When $X$, $Y$ are both intervals or $\bS^1$, we say a topological conjugacy $f$ is positive if it is orientation preserving. 

 Notice that a (pre-)Kupisch function $\kappa$ on $\RR$ or $\RR^{\geq 0}$ gives rise to a dynamical system $(\RR,K(t))$ or $(\RR^{\geq 0}, K(t))$, where $K(t)=\kappa(t)+t$ is the self-map.
And a Kupisch function $\kappa$ induces a circle map  
\begin{eqnarray*}
\bar K:\bS^1&\to& \bS^1 \\
 (e^{2\pi i t})&\mapsto& e^{2\pi i (\kappa(t)+t)},
\end{eqnarray*}
hence a system $(\bS^1, \bar K)$.
One can check that a (pre-)Kupisch function $\lambda=f_*\kappa$ is a push-forward by some orientation preserving homeomorphism $f$ if and only if $f$ is a positive topological conjugacy between the induced dynamical systems. Thus  Theorem \ref{R equiv 2} and Theorem \ref{S1 equiv 2} says that classifying categories $(\bS^1,\kappa)\pwf$ and $(\RR,\kappa)\pwf$ is equivalent to classifying their corresponding dynamical system up to positive topologically conjugacy. See \cite{KH} for more details about circle dynamics.

\section{Continuous and Discrete Nakayama Representations}

In this section, we discuss the relation between discrete and continuous representations. We show that the module category of any discrete Nakayama algebra can be embedded in $(\bS^1,\kappa)\pwf$ for some Kupisch function $\kappa$.

Recall that a basic connected Nakayama algebra $A$ is called {\bf linear} if its $\Ext$-quiver is ${\bf A}_n$: 
$$
\begin{tikzcd}
&0\ar[r]&1\ar[r]&2\ar[r]&\cdots\ar[r]&n-1	
\end{tikzcd}
$$
and {\bf cyclic} if its $\Ext$-quiver is ${\bf C}_n$:

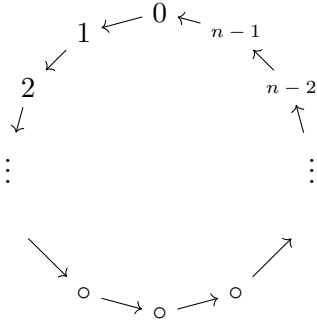
$$
\begin{tikzpicture}[->]
\node(A) at (0,2) {$0$};
\node(B) at (1,1.73) {\tiny $n-1$};
\node (C) at (1.73,1) {\tiny $n-2$};
\node (D) at (2,0) {$\vdots$};
\node(E) at (0,-2) {$\circ$};
\node (E1) at (-1,-1.73) {$\circ$};
\node (E2) at (1,-1.73) {$\circ$};
\node (F) at (-2,0) {$\vdots$};
\node (G) at (-1,1.73) {$1$};
\node(H) at (-1.73,1) {2};
\draw(D)--(C);
\draw(C)--(B);
\draw(B)--(A);
\draw(A)--(G);
\draw(G)--(H);
\draw(H)--(F);
\draw(E1)--(E);
\draw(E)--(E2);
\draw (-1.73,-1)--(E1);
\draw (E2)--(1.73,-1);	
\end{tikzpicture}
$$

Let $A$ be a basic connected Nakayama algebra of rank $n$ (i.e. $A$ has $n$ isomorphism classes of simple modules). The Kupisch series $(l_0,\cdots, l_{n-1})$ is a sequence with $l_i=l(P_i)$ the length of the indecomposable projective module $P_i$. By defninition, $A$ is a (connected) linear Nakayama algebra if and only if $l_{n-1}=1$.  Since any ${\bf A}_n$ representation $M$ can be viewed as a ${\bf C}_n$ representation with $M(n-1)\to M(0)$ being zero, representations of any Nakayama algebra are representations on ${\bf C}_n$.

\begin{defn}\label{associate kupish}
Let $A$ be a basic connected Nakayama algebra with a Kupisch series $(l_0,l_1,\cdots, l_{n-1})$.
We define the \textbf{associated Kupisch function $\kappa_A$} as follows. For any $k\in\mathbb Z$,
$$
\kappa_A(t):=  \frac{i+l_{i}}{n}-t+k,\ \ \frac{i}{n}+k\leq t< \frac{i+1}{n}+k, \text{for $0\leq i<n$}.
$$
\end{defn}

The relation between $A$-modules and $\bS^1$-representations compatible with $\kappa_A$ is described in Lemma \ref{lem:GM} and Remark \ref{rmk:GM}. Notice that the associated Kupisch function $\kappa_A$ has no separation points.

\begin{exm}
Let $A$ be a Nakayama algebra with Kupisch series $(3,3,2)$. Then the associated Kupisch function $\kappa_A$ on $[0,1)$ is:
$$
\kappa_A(t)=\begin{cases}  1-t &0\leq t<\frac{1}{3}\\
\frac{4}{3}-t& \frac{1}{3}\leq t<1 \end{cases}
$$
\end{exm}

The rest of this section is dedicated to proving the following theorem.

\begin{thm}\label{embedding thm}
Let $A$ be a basic connected Nakayama algebra with a Kupisch series $(l_0,l_1,\cdots, l_{n-1})$ and $\kappa_A$ its associated Kupisch function. Then there is an exact embedding $F: A\modd\to (\bS^1, \kappa_A)\pwf$ that preserves projective objects. Moreover, if $A$ is a linear Nakayama algebra, then there is an exact embedding $L:A\modd\to (\RR, \kappa_A)\pwf$, which preserves projective objects.
\end{thm}

For the convenience of showing functoriality, we need to describe the image of this functor on objects first.
Let $\cM=\add\{\overline M_{(\frac{a}{n},\frac{b}{n}]}\in(\bS^1,\kappa_A)\pwf \mid a,b\in\mathbb N, a<n\}$ be the full subcategory of $(\bS^1,\kappa_A)\pwf$ consisting of direct sums of string modules of the form $\overline M_{(\frac{a}{n},\frac{b}{n}]}$. 

For $0\leq i < n$, denote by $p_i=(\frac{i+1}{n} \mod\mathbb Z)\in \mathbb R/\mathbb Z\cong\bS^1$. 
For each string module $\overline M_{(\frac{a}{n},\frac{b}{n}]}\in\cM$, we assign a representation $G\overline M_{(\frac{a}{n},\frac{b}{n}]}=(X_i,\varphi_i)$ of the quiver ${\bf C}_n$ as the following: $X_{i}=  \overline M_{(\frac{a}{n},\frac{b}{n}]}(p_i)$ for $0\leq i<n$ and $\varphi_i:X_{i}\to X_{i+1}$ is given by $\varphi_i= \overline M_{(\frac{a}{n},\frac{b}{n}]}(g_{p_ip_{i+1}})$, where $g_{p_ip_{i+1}}$ is a generating morphism on $\bS^1$.

\begin{lem}\label{lem:GM}
Let $\overline M_{(\frac{a}{n},\frac{b}{n}]}\in\cM$, then $M=G\overline M_{(\frac{a}{n},\frac{b}{n}]}$ is an indecomposable $A$-module with
$\Top M\cong S_a$ and length $l(M)=b-a$.
\end{lem}

\begin{proof}
According to the definition of string module $\overline M_{ (\frac{a}{n},\frac{b}{n}] }$ as in Section \ref{sec:bS persistence}, the representation $M=(X_i,\varphi_i)$ can be explicitly written in the following way.

Let $U= (\frac{a}{n},\frac{b}{n}]$, $x=p_{i+1}$, and $y=p_{i+2}$.

The vector space $X_k$ has a basis $p_U^{-1}(x)=\{b_{i,x}\}_{i=1}^{|p_U^{-1}(x)|}$ such that $b_{i,x}< b_{i+1,x}$ for all $i$.

The linear transformation $\varphi_k=\overline{M}_U(g_{xy}):X_k\to X_{k+1}$ is defined on the basis of $X_k$ as:
\begin{displaymath}
    \varphi_k(b_{i,x}) = \begin{cases}
        b_{j,y} & \exists b_{j,y}\in p_U^{-1}(y) \text{ such that } 0\leq b_{j,y}-b_{i,x} < 1 \\ 
        0 & \text{otherwise}.
    \end{cases}
\end{displaymath}

Thus, $M=(X_i,\varphi_i)$ is a string module on ${\bf C}_n$ with $\Top M\cong S_a$ and $l(M)=b-a$. 
To see $M$ is an $A$-module, we just need to check $l(M)\leq l_a$ for the Kupisch series $(l_0,l_1,\cdots, l_a,\cdots, l_{n-1})$. But this follows immediately since $b-a\leq nK_A(a)-a\leq l_a$.
\end{proof}

\begin{rmk}\label{rmk:GM} We mention some important $A$-modules corresponding to the assignment $G$.
\begin{enumerate}
 \item Each $G\overline M_{(\frac{i}{n},\frac{i+1}{n}]}$ is isomorphic to the simple module $S_i$.
 \item Each $G\overline M_{(\frac{i}{n},K_A(\frac{i}{n})]}$ is isomorphic to the projective module $P_i$. 
   \end{enumerate}
\end{rmk}

Furthermore, the assignment $G$ defines an additive functor: a morphism $f:\overline M\to \overline N$ in $\cM$ gives rise to an $A$-module homomorphism $Gf:G\overline M=(X_i,\varphi_i)\to G\overline N=(Y_i,\psi_i)$ by taking $Gf_i:X_i\to Y_i$ to be $f(p_{i})$. 

\begin{lem}\label{lem:G eqv}
The additive functor $G:\cM\to A\modd$ is an equivalence.
\end{lem}

\begin{proof}
We need to show that $G$ is fully-faithful and dense.

Let $f: \overline M\to \overline N$ be a morphism in $\cM$ such that $Gf=0$. By definition, $f(p_i)=0$ for all $i$. For any $x,y\in (\frac{i}{n},\frac{i+1}{n}]$ and for all $i$, $\overline M(g_{xy})$ and $\overline N(g_{xy})$ are identity maps. Since $f(p_i)\overline{M}(g_{xp_i})=\overline{N}(g_{xp_i})f(x)$, it follows that $f(x)=f(p_i)=0$ for for all $x\in (\frac{i}{n},\frac{i+1}{n}]$ and for all $i$. Hence, $f=0$. 

Conversely, for any homomorphism $g=(g_i):G\overline M=(X_i,\varphi_i)\to G\overline N=(Y_i,\psi_i)$, we define a natural transformation $f:\overline M\to \overline N$ such that $f(x)=g_i$ for $x\in(\frac{i}{n},\frac{i+1}{n}]$.
It satisfies $Gf=g$ and, therefore, $G:\Hom_\cM(\overline M,\overline N)\to\Hom_A(G\overline M,G\overline N)$ is bijective.

Finally, since $A$ is a basic connected Nakayama algebra, any indecomposable $A$-module can be uniquely determined by its top and length up to isomorphism. Assume $X$ is an indecomposable $A$-module with $\Top X\cong S_a$ and $l(X)=l$. Then, according to Lemma \ref{lem:GM}, $G(\overline M_{(\frac{a}{n},\frac{a+l}{n}]})\cong X$. Therefore, $G$ is dense.
\end{proof}

\begin{cor}
Let $A$ be a basic connected Nakayama algebra with Kupisch series $(l_0,l_1,\cdots, l_{n-1})$, $\kappa_A$ be its associated Kupisch function, and $\overline P=\bigoplus\limits_{i=0}^{n-1} \overline M_{(\frac{i}{n}, K_A(\frac{i}{n})]}$ be in $(\bS^1,\kappa_A)\pwf$. Then $A\cong\End_{\bS^1}(\overline P)$. 
\end{cor}
\begin{proof}
 Since $G\overline M_{(\frac{i}{n}, K_A(\frac{i}{n})]}=P_i$ are the indecomposable projective $A$-modules, it follows that $\displaystyle A\cong \End_A\left(\bigoplus\limits_{i=0}^{n-1} P_i\right)\cong \End_{\bS^1}\left(\bigoplus\limits_{i=0}^{n-1} \overline M_{(\frac{i}{n}, K_A(\frac{i}{n})]}\right)$.
\end{proof}

{\bf Proof of Theorem \ref{embedding thm}:} The embedding $\iota: \cM\to (\bS^1,\kappa_A)\pwf$ is exact by construction.  Therefore, there is an exact embedding $F=\iota\circ G^{-1}:A\modd\to (\bS^1,\kappa_A)\pwf$.
Each indecomposable projective $A$-module $P(a)$ has $\Top P(a)\cong S_a$ and $l(P(a))=l_a$. Thus, $F(P(a))\cong \overline M_{(\frac{a}{n},\frac{a+l_a}{n}]}= \overline M_{(\frac{a}{n},K(\frac{a}{n})]}$, which is a projective object in $(\bS^1,\kappa_A)\pwf$. 

When $A$ is a linear Nakayama algebra, it follows that $l_{n-1}=1$. Hence $K_A(t)\leq 1$ for all $0\leq t<1$. So for each string module $\overline M_U$ in $\cM$, we have $U\subseteq (0,1]$ and $\displaystyle\Hom_{\cM}(\overline M_U,\overline M_V)=\bigoplus_{i\in\ZZ}\Hom_{\RR}(M_U,M_{V+i})=\Hom_{\RR}(M_U,M_{V})$. Therefore $\iota:\cM\to (\RR,\kappa_A)$ sending $\bigoplus\overline M_{U}$ to $\bigoplus M_U$ is a fully-faithful embedding and $L=\iota\circ G^{-1}:A\modd\to (\RR,\kappa_A)\pwf$ is the desired exact embedding.
\hfill $\qed$

Theorem \ref{embedding thm} says that the module category of any basic connected Nakayama algebra can be regarded as an exact abelian subcategory of $(\bS^1,\kappa)\pwf$ or $(\RR,\kappa)\pwf$ for some Kupisch function $\kappa$.
Therefore, we call $(\bS^1,\kappa)\pwf$ or $(\RR,\kappa)\pwf$ a category of continuous Nakayama representations. 

Continuous Nakayama representations have composition series in the sense of \cite{HR2}. In particular, it is shown in \cite{HR2} that indecomposable representations of $\RR$ with straight orientation (and thus string representations of $\bS^1$ with cyclic orientation) are uniserial.
Thus, an indecomposable contiuous Nakayama representation of $(\RR,\kappa)\pwf$ or of $(\bS^1,\kappa)\pwf$ is also uniserial, for some (pre-)Kupisch function $\kappa$.
Although categories of both discrete and continuous Nakayama representations have uniserial indecomposable objects, the homological properties of continuous Nakayama representations are sometimes quite different. We conclude with some interesting examples, which demonstrates the difference of some homological properties between categories of discrete and continuous Nakayama representations.

\begin{exm}
Consider $\kappa(t)=\begin{cases}0.1 & n\leq t\leq n+\frac{1}{2},\, n \in\ZZ \\
1-(t-n) &n+\frac{1}{2}<t<n+1, n\in\ZZ\end{cases}$

The simple representation $\overline M_{[0,0]}=\overline M_{[1,1]}$ does not have an injective envelope.
This is because, if there was an injective envelope $\overline M_{[1,1]}\to \overline E$, then $\overline E$ must be indecomposable and of the form $\overline E=\overline M_{[b,1]}$.  But since any inclusion $\overline{M}_{[1,1]}\hookrightarrow\overline{M}_{[\frac{1}{2}+\e,1]}$ ($0<\e<\frac{1}{2}$) cannot factor through $\overline{M}_{[1,1]}\hookrightarrow\overline{M}_{[\frac{1}{2}+\frac{\e}{2},1]}$, there is no such an injective envelope.
\end{exm}

\begin{exm}\label{findim}
 Consider the category $(\bS^1,\kappa)\pwf$, where for $t\in(0,1]$
$$
K(t)=\kappa(t)+t=\begin{cases}  \frac{1}{n-1}+1 &\frac{1}{n+1}\leq t <\frac{1}{n}, \text{\ for\ }  n\geq 4\\
 \frac{3}{2} &\frac{1}{4}\leq t\leq 1.
 \end{cases}
$$
Every string module $\overline M_U$ has finite projective dimension. But $\pd \overline M_{(\frac{1}{2k+1},\frac{1}{2k}]}=2k-1$ for $k\in\mathbb N^{>0}$. Hence $\displaystyle\pd \left(\bigoplus_{k=1}^{\infty} M_{(\frac{1}{2k+1},\frac{1}{2k}]}\right)=\infty$ and $\sup\{\pd \overline M_U\mid \pd \overline M_U<\infty\}=\infty$.
\end{exm}

\end{document}